\documentclass{amsart}

\usepackage[T1]{fontenc}
\usepackage[utf8]{inputenc}
\usepackage{amssymb,amsfonts}
\usepackage{setspace}
\usepackage[margin=1.25in]{geometry}

\newtheorem{theorem}{Theorem}[section]
\newtheorem{lemma}[theorem]{Lemma}
\newtheorem{proposition}[theorem]{Proposition}
\newtheorem{example}[theorem]{Example}
\newtheorem*{claim}{Claim}
\newtheorem{claimn}{Claim}

\theoremstyle{remark}
\newtheorem{exercise}[theorem]{Exercise}

\newcommand{\cgF}{\mathcal{F}}
\newcommand{\cgG}{\mathcal{G}}
\newcommand{\cgH}{\mathcal{H}}

\newcommand{\Inc}{\operatorname{Inc}}
\newcommand{\Max}{\operatorname{Max}}
\newcommand{\Min}{\operatorname{Min}}
\newcommand{\Crit}{\operatorname{Crit}}

\newcommand{\FCBP}{\operatorname{FCBP}}
\newcommand{\LCBP}{\operatorname{LCBP}}
\newcommand{\FEBP}{\operatorname{FEBP}}
\newcommand{\LEBP}{\operatorname{LEBP}}
\newcommand{\DFCL}{\operatorname{DFCL}}
\newcommand{\DFEL}{\operatorname{DFEL}}
\newcommand{\DLCF}{\operatorname{DLCF}}
\newcommand{\DLEF}{\operatorname{DLEF}}

\newcommand{\MI}{\mathbb{MI}}
\newcommand{\NI}{\mathbb{MSI}}
\renewcommand{\MR}{\mathbb{MR}}
\newcommand{\INR}{\mathbb{INR}}
\newcommand{\MINR}{\mathbb{MINR}}
\newcommand{\NINR}{\mathbb{MSINR}}
\newcommand{\CDP}{\mathbb{C}_{D3}}

\newcommand{\CZ}{\mathbb{C}_{D3}}
\newcommand{\mspread}{\operatorname{maxspread}}

\title[The Graph of Critical Pairs 
  of a Crown]{The Graph of Critical Pairs 
  of a Crown}

\author[BARRERA-CRUZ]{Fidel Barrera-Cruz}

\address{School of Mathematics\\
  Georgia Institute of Technology\\
  Atlanta, Georgia 30332 (Barrera-Cruz, Smith, Taylor, Trotter)}

\email{\{fidelbc, heather.smith, trotter\}@math.gatech.edu; libbytaylor@gatech.edu}

\author[GARCIA]{Rebecca Garcia}

\address{Department of Mathematics and Statistics\\
 Sam Houston State University\\
 Huntsville, Texas 77341 (Garcia)}

\email{mth\_reg@shsu.edu}

\author[HARRIS]{Pamela Harris}

\address{Department of Mathematics and Statistics\\
 Williams College\\
 Williamstown, Massachusetts 01267 (Harris)}

\email{pamela.e.harris@williams.edu}

\author[KUBIK]{Bethany Kubik}

\address{Department of Mathematics and Statistics\\
 University of Minnesota Duluth\\
 Duluth, Minnesota 55812 (Kubik)}

\email{bakubik@d.umn.edu}

\author[SMITH]{Heather Smith}

\author[TALBOTT]{Shannon Talbott}

\address{Department of Mathematics and Computer Science\\
 Moravian College\\
 Bethlehem, Pennsylvania 18018 (Talbott)}

\email{talbotts@moravian.edu}
 
\author[TAYLOR]{Libby Taylor}

\author[TROTTER]{William T. Trotter}

\date{August 22, 2017}



\begin{document}
%\linenumbers

\begin{abstract}
There is a natural way to associate with a poset $P$ a hypergraph
$\cgH$, called the hypergraph of critical pairs, so that
the dimension of $P$ is exactly equal to the chromatic number 
of $\cgH$.  The edges of $\cgH$ have variable sizes, but it is
of interest to consider the graph $G$ formed by the edges of $\cgH$ that have size~$2$. The chromatic number of $G$ is
less than or equal to the dimension of $P$ and the difference
between the two values can be arbitrarily large. Nevertheless,
there are important instances where the two parameters are 
the same, and we study one of these in this paper.  Our focus is 
on a family $\{S_n^k:n\ge3, k\ge0\}$ of height two posets
called crowns.  We show that the chromatic 
number of the graph $G_n^k$ of critical pairs of the crown 
$S_n^k$ is the same as the dimension of $S_n^k$, which is known
to be $\lceil 2(n+k)/(k+2)\rceil$.  In fact,
this theorem follows as an immediate corollary to the stronger 
result: The independence number of $G_n^k$ is $(k+1)(k+2)/2$.
We obtain this theorem as part of a comprehensive analysis of
independent sets in $G_n^k$ including the determination of the
second largest size among the maximal independent sets, 
both the reversible and non-reversible types.
\end{abstract}

\maketitle

\section{Introduction}

In this paper, we are concerned with a family of posets
introduced in~\cite{bib:Trot-Snk} called \textit{crowns}. 
For a pair $(n,k)$ of integers with $n\ge3$ and $k\ge0$,
the \textit{crown} $S_n^k$ is the height~$2$ poset whose ground
set\footnote{In the original paper~\cite{bib:Trot-Snk}, $A$ is
the set of maximal elements and $B$ is the set of minimal elements.
Here these roles are reversed, a practice that
is consistent with modern research papers in the area.  Additionally, we use an 
entirely equivalent, but slightly modified labeling in order to considerably simplify the arguments to follow.} is
$A\cup B$ where $A=\Min(P)=\{a_1,a_2,\dots,a_{n+k}\}$ 
and $B=\Max(P)=\{b_1,b_2,\dots,b_{n+k}\}$.  Furthermore, $a_i$ is 
incomparable with $b_j$ when $j$ belongs to the interval $\{i,i+1,i+2,
\dots,i+k\}$ and $a_i<b_j$ in $S_n^k$ when $j$ is not in this 
interval.  Of course, this definition must be interpreted cyclically, 
i.e., $n+k+1=1$, $n+k+2=2$, etc. 

As will be detailed later, the critical pairs of $S_n^k$ are precisely the incomparable pairs of the form $(a_i, b_j)$. The associated graph of critical pairs, $G_n^k$, has a vertex for each critical pair and an edge between $(a_i, b_j)$ and $(a_k, b_\ell)$ if and only if $a_i<b_\ell$ and $a_k<b_j$ in $S_n^k$. 

Originally, our goal was to prove the following theorem, a result
that was conjectured by Garcia, Harris, Kubik, and Talbott~\cite{bib:GHKT}
and launched the research collaboration represented by this
manuscript.

\begin{theorem}\label{thm:chi=dim}
Let $(n,k)$ be a pair of integers with $n\ge3$ and $k\ge0$, and
let $G_n^k$ be the graph of critical pairs associated with the
crown $S_n^k$.  Then the chromatic number of $G_n^k$ is equal to 
the dimension of $S_n^k$.
\end{theorem}

As we explain in the next section, the preceding
theorem follows as an immediate corollary Theorem~\ref{con:alpha(G)}.

\begin{theorem}\label{con:alpha(G)}
Let $(n,k)$ be a pair of integers with $n\ge3$ and $k\ge0$, and
let $G_n^k$ be the graph of critical pairs associated with the
crown $S_n^k$.  Then the maximum size of an independent set
in $G_n^k$ is $(k+1)(k+2)/2$.
\end{theorem}

We first managed to prove Theorem~\ref{thm:chi=dim} \textit{without}
resolving Theorem~\ref{con:alpha(G)}, and we explain in Section~\ref{sec:close} how this
was achieved. However,
in preparing early versions of this paper, we continued to
build on our understanding of the properties of independent sets in
$G_n^k$, and as a result, we were able to settle Theorem~\ref{con:alpha(G)}. In fact, this proof emerges as a minor
detail extracted from a comprehensive body of results concerning
independent sets in $G_n^k$.  

First, independent sets in $G_n^k$ will be classified as being
one of two types: reversible or non-reversible. A reversible set is
always independent, but in general, there are non-reversible sets that
are also independent.  A subset of a reversible set is reversible, and
a subset of an independent set is independent, so there are natural
notions of a maximal reversible set and a maximal independent set.  Among
the reversible sets, a special family is defined and members of this
family are called canonical reversible sets.
All canonical reversible sets have size $(k+1)(k+2)/2$ which is the maximum size for a reversible set, and the maximum size for an independent set.
Here is the first of our three main theorems.

\begin{theorem}\label{thm:main-1}
Let $(n,k)$ be a pair of integers with $n\ge3$ and $k\ge0$.
Then the following statements hold.

\begin{enumerate}
\item If $S$ is a maximal reversible set in $G_n^k$,
then $S$ is a maximal independent set in $G_n^k$.
\item If $n>k$ and $S$ is a maximal reversible set in $G_n^k$,
then $S$ is a canonical reversible set in $G_n^k$.
\item If $S$ is a maximal reversible set in $G_n^k$ and
$S$ is not a canonical reversible
set, then $n\le k$ and \[|S|\le \frac{(k+1)(k+2)}{2}-\frac{n(n-1)}{2}+1.\]
\end{enumerate}
\end{theorem}

Although it is possible to completely determine 
the family of all maximal reversible sets, our focus is on
finding the second largest size such sets can attain.  This detail
is critical to subsequent arguments.

To set the stage for the study of independent, non-reversible
sets, we prove the following elementary lemma in Section~\ref{sec:part-1}.

\begin{lemma}\label{lem:INR-exists}
Let $(n,k)$ be a pair of integers with $n\ge3$ and $k\ge0$.
Then there is an independent non-reversible set in $G_n^k$
if and only if $n\le 2k$.
\end{lemma}

As we begin to investigate independent, non-reversible sets,
it will quickly become clear that there is a natural division
into two ranges:  $n\le k$ and $k<n\le 2k$.
Independent, non-reversible sets in the range $k<n\le 2k$ are
relatively simple, and we completely determine the family of all 
maximal independent, non-reversible sets.  Accordingly, 
we also establish the following upper bound and show
that it is best possible.  This result is the second
of our three main theorems.

\begin{theorem}\label{thm:main-2}
Let $(n,k)$ be a pair of integers with $n\ge3$ and $k\ge0$.
If $k<n\le 2k$ and $S$ is an independent, non-reversible
set in $G_n^k$, then \[|S|\le 2+\frac{(2k+2-n)(2k+1-n)}{2}.\]
\end{theorem}

In particular, when $k<n\le 2k$, there is a unique (up to isomorphism) independent, 
non-reversible set of size~$2+(2k+2-n)(2k+1-n)/2$.

The situation when $n\le k$ is considerably more complex, and we are unable
to determine the family of all maximal independent, non-reversible
sets.  However, we provide the following upper bound on
their size in our third major theorem.

\begin{theorem}\label{thm:main-3}
Let $(n,k)$ be a pair of integers with $n\ge3$ and $k\ge0$.
If $n\le k$ and $S$ is an independent, non-reversible
set in $G_n^k$, then \[|S|\le \frac{(k+1)(k+2)}{2}+2-n.\]  
\end{theorem}

We show that the inequality in Theorem~\ref{thm:main-3} is best
possible.  In fact, we have been able to completely characterize
the family of all extremal examples.  Due to the length of the details for the proof, we have
elected to present a single representative example in this paper.

It is worth noting that the inequality from Theorem~\ref{thm:main-3}
holds for all pairs $(n,k)$, since when $k<n\le 2k$ we have:
\begin{align*}
2+\frac{(2k+2-n)(2k+1-n)}{2}&=\frac{(k+1)(k+2)}{2}+2-n-\frac{(3k-n)(n-k-1)}{2}\\
                &\le\frac{(k+1)(k+2)}{2}+2-n.
\end{align*}
However, this inequality is strict when $n>k+1$.

Once these theorems have been proved, it is simply a remark that for all pairs $(n,k)$ with $n\ge3$ and $k\ge0$, the maximum 
size of an independent set in $G_n^k$ is $(k+1)(k+2)/2$,
an observation which is enough to show that
$\dim(S_n^k)=\chi(G_n^k)$.  However, we have proved much more since we have identified the
canonical reversible sets as the maximum size independent sets, and
we have found the maximum size for all other maximal independent
sets, for both the reversible and non-reversible cases.
  
The remainder of this paper is organized as follows. In the next section,
we provide notation and terminology together with a concise summary of
background material to motivate this line of research.
In Section~\ref{sec:reversible}, we study reversible sets and
prove Theorem~\ref{thm:main-1}.  Independent, non-reversible
sets are more complex and Section~\ref{sec:part-1} is
an introductory section in which essential proof techniques are developed and the proof of 
Lemma~\ref{lem:INR-exists} is given.  Sections~\ref{sec:part-2}
and~\ref{sec:part-3} are devoted to the proofs of the inequalities
in Theorems~\ref{thm:main-2} and~\ref{thm:main-3}, respectively.  

In Section~\ref{sec:close}, we provide details on how
Theorem~\ref{thm:chi=dim} can be proved directly without Theorem~\ref{con:alpha(G)}.  These techniques are of independent interest, even if the conclusion
is attainable by the results presented earlier in the paper.

We close in Section~\ref{sec:close} with some comments on 
challenging open problems that remain.

\section{Notation, Terminology, and Background Material}

We assume that the reader is familiar with basic notation and terminology
for partially ordered sets (here we use the short term poset), including:
chains and antichains, minimal and maximal elements, linear extensions, and comparability graphs.  While we are also assuming
some level of familiarity with the concept of dimension for posets,
there are a number of recent 
papers, ~\cite{bib:FelTro,bib:JMMTWW, bib:TroWan}, each of which contains a more complete discussion of the necessary
background material. A 
comprehensive treatment is given in~\cite{bib:Trot-Book}, and a survey of 
combinatorial aspects of posets is given in~\cite{bib:Trot-Handbook}, 
so we include here only the essential definitions.

A non-empty family $\cgF=\{L_1,L_2,\dots,L_d\}$ of linear extensions
of a poset $P$ is called a \textit{realizer} of $P$ when $x\le y$ in $P$
if and only if $x\le y$ in $L_i$ for each $i\in\{1,2,\dots,d\}$.  
As defined by Dushnik and Miller~\cite{bib:DusMil}, the
\textit{dimension} of $P$, denoted $\dim(P)$, is the least
positive integer $d$ for which $P$ has a realizer of size~$d$. 

When $P$ is a poset, we let $\Inc(P)$ be the set of all pairs 
$(x,y)\in P\times P$ with $x$ incomparable to $y$ in $P$. 
Clearly, a non-empty family $\cgF$ of linear extensions of $P$
is a realizer if and only if for every $(x,y)\in\Inc(P)$, $x>y$ in $L$ for some $L\in\cgF$ due to the symmetry of pairs in $\Inc(P)$.

A subset $R\subset\Inc(P)$ is said to be \textit{reversible} when
there is a linear extension $L$ of $P$ with $x>y$ in $L$ for all
$(x,y)\in R$.  Accordingly, when $\Inc(P)\neq\emptyset$,
$\dim(P)$ is the least $d$ for which there is a covering
$\Inc(P)=R_1\cup R_2\cup\dots\cup R_d$ with $R_i$ reversible for
each $i\in\{1,2,\dots,d\}$.

An indexed set $C=\{(x_\alpha,y_\alpha): 1\le \alpha\le m\}\subseteq\Inc(P)$
of incomparable pairs in $P$
is called an \textit{alternating cycle} of size~$m$ when\footnote{Most
authors require that $x_\alpha\le y_{\alpha+1}$ in $P$ in defining an
alternating cycle.  Our equivalent formulation is another choice that
simplifies arguments to follow.}
$x_\alpha\le y_{\alpha-1}$ in $P$, for all $\alpha\in\{1,2,\dots,m\}$. (Subscripts are interpreted cyclically so that $x_1\le y_m$ in $P$.)  An 
alternating cycle is \textit{strict} when $x_\alpha\le y_\beta$ if and 
only if $\beta=\alpha-1$.  In a strict alternating cycle, the set
$\{x_\alpha:1\le \alpha\le m\}$ is an $m$-element antichains in $P$ as is  $\{y_\alpha:1\le \alpha\le m\}$.

Although the proof of the following lemma, first presented
in~\cite{bib:TroMoo}, is elementary, the basic ideas behind
this result have proven over time to be very important.

\begin{lemma}\label{lem:AC}
Let $P$ be a poset and let $S\subseteq\Inc(P)$. Then the following 
statements are equivalent.

\begin{enumerate}
\item $S$ is not reversible.
\item $S$ contains an alternating cycle.
\item $S$ contains a strict alternating cycle.
\end{enumerate}
\end{lemma}

We note the following property of alternating cycles:\quad
If $C=\{(x_\alpha,y_\alpha):1\le \alpha\le m\}$ is an alternating cycle, 
but is not strict, then there is a proper subset of the pairs in $C$ 
which (after a relabeling) forms a strict alternating cycle.

A pair $(x,y)\in\Inc(P)$ is called a \textit{critical pair}
if (1)~$z<y$ in $P$ whenever $z<x$ in $P$ and (2)~$w>x$ whenever $w>y$ in
$P$.  We let $\Crit(P)$ denote the set of all critical pairs.
Interest in critical pairs is rooted in the well known property
that a non-empty family $\cgF$ of linear extensions is a realizer 
if and only if for every pair $(x,y)\in\Crit(P)$, there is some 
$L\in \cgF$ with $x>y$ in $L$.  When $\Inc(P)\neq\emptyset$, $\dim(P)$ is the
least positive integer $d$ for which there is a covering
$\Crit(P)=R_1\cup R_2\cup\dots\cup R_d$ with $R_i$ reversible for
each $i\in\{1,2,\dots,d\}$.

Although the origins can be traced back to earlier papers, it seems that
the first concrete formulation of the following concept appears
in~\cite{bib:KelTro}:  Given a poset $P$ for which $\Inc(P)\neq\emptyset$, we 
can associate with $P$ a hypergraph $\cgH$ and a graph $G$ of critical pairs defined as follows.
Both $\cgH$ and $G$ have the set $\Crit(P)$ of critical pairs as their vertex 
set.  In $\cgH$, a set $E$ of critical pairs is an edge when $E$ is not 
reversible, but every proper subset of $E$ is reversible. The edge set of the graph $G$ is just the set of all edges in $\cgH$ 
which have size~$2$.  In view of the 
remarks made immediately above, $\dim(P)=\chi(\cgH)$ when $\Inc(P)\neq\emptyset$, where $\chi(H)$ is the minimum $k$ such that there is a $k$-coloring of the vertices of $H$ with no monochromatic edge.  On the other hand, we only have the inequality
$\dim(P)\ge\chi(G)$.  

In~\cite{bib:FelTro}, an infinite sequence $\{P_n:n\ge 1\}$ of posets is 
constructed such that the dimension of $P_n$ grows
exponentially with $n$ while the chromatic number of the graph of
critical pairs grows linearly with $n$.  Accordingly, the inequality
$\dim(P)\ge\chi(G)$ can be far from tight. Nevertheless, it is of interest
to investigate conditions which cause $\dim(P)$ to be equal to $\chi(G)$.

The research we report concerns a well studied class of
posets called crowns.  The first use of the term ``crown'' in reference
to a class of posets is in~\cite{bib:BaFiRo}, where it is applied only
to posets in the subfamily $\{S_3^k:k\ge0\}$.  These posets are all
$3$-irreducible, i.e., they have dimension~$3$, but the removal of any
point lowers the dimension to~$2$.  This special case plays an
important role in the on-line notion of dimension (see~\cite{bib:FeKrTr} and~\cite{bib:KiMcTr}).  Also, the family of crowns 
includes the \textit{standard examples}.  These are the posets
in the family $\{S_n^0:n\ge3\}$.  For each $n\ge3$, the crown
$S_n^0$ is $n$-irreducible, and in the literature, the notation
for the standard example $S_n^0$ is usually abbreviated as $S_n$.

The fact that $\dim(S_3^k)=3$ for all $k\ge0$ and $\dim(S_n^0)=n$ for
all $n\ge3$ suggests that there may be a function
$f(d)$ such that if the maximum degree in the comparability graph of $P$ 
is $d$, then the dimension of $P$ is at most $f(d)$.  In fact, these
examples suggest that it might even be true that $f(d)=d+1$.  But for
many years, it was not even known whether the function $f(d)$ was well
defined.  However, F\"{u}redi and Kahn~\cite{bib:FurKah} proved that
$f(d)=O(d\log^2 d)$ and Erd\H{o}s, Kierstead and Trotter~\cite{bib:ErKiTr}
proved that $f(d)=\Omega(d\log d)$. So the original interest in the
family of crowns was to see if these posets shed further light on
the problem of dimension versus maximum degree in the comparability
graph.

Most authors extend the notion of the standard example $S_n$ to the
value $n=2$, i.e., $S_2$ is the poset of height~$2$ with minimal
elements $\{a_1,a_2\}$, maximal elements $\{b_1,b_2\}$, and
$a_i<b_j$ in $S_2$ if and only if $i\neq j$.  
The standard example $S_2$ is just $\mathbf{2}+\mathbf{2}$, the
disjoint sum of two $2$-element chains, with all points of one chain
incomparable with all points in the other.
Posets which exclude $S_2=\mathbf{2}+\mathbf{2}$ are called \textit{interval
orders}, since if $P$ is such a poset,  there is family
$\cgF=\{[c_x,d_x]:x\in P\}$ of non-degenerate closed intervals
of the real line so that $x<y$ in $P$ if and only
if $d_x<c_y$ in the reals.  The class of interval orders has
been studied extensively in the literature. See~\cite{bib:Fish} 
and~\cite{bib:Trot-LMS} for results and references.  Note that $S_2^0$ has 
dimension~$2$, but it is not $2$-irreducible.

Based on the examples in the families $\{S_3^k:k\ge0\}$ and 
$\{S_n^0:n\ge3\}$, it was originally thought that it might
be possible that $\dim(S_n^k)=n$ for all pairs $(n,k)$ with
$n\ge3$ and $k\ge0$.  Some small cases not belonging to these
families were worked out by hand, leading first to the conclusion
that $\dim(S_4^1)=4$ which fit the suspected pattern. But subsequently,
it was shown that $\dim(S_4^2)=3$, so the pattern does not hold
in general.  These observations then motivated an attack on
finding the general form for the dimension of the crown $S_n^k$,
which led eventually to the following formula, given
in~\cite{bib:Trot-Snk}.

\begin{theorem}\label{thm:Snk}
Let $(n,k)$ be a pair of integers with $n\ge3$ and $k\ge0$.  Then
$\dim(S_n^k)= \lceil 2(n+k)/(k+2)\rceil$.
\end{theorem}

Clearly, the critical pairs in the crown $S_n^k$ are just the pairs 
$(a,b)\in A\times B$ with $a$ incomparable to $b$ in $S_n^k$.  We denote the
set of such pairs by $\Inc(A,B)$.  There are $(n+k)(k+1)$ critical 
pairs in $\Inc(A,B)$, so the inequality
$\dim(S_n^k)\ge 2(n+k)/(k+2)$ is an immediate consequence of
the following lemma, which appears (with different notation) on page~$92$ 
in~\cite{bib:Trot-Snk}.  

\begin{lemma}\label{lem:max-reverse-cps}
The maximum number of critical pairs that can be reversed by a 
linear extension of the crown $S_n^k$ is $(k+1)(k+2)/2$.
\end{lemma}

In retrospect, it is fair to say that the argument 
presented in~\cite{bib:Trot-Snk} is incomplete.  
However, the lemma also appears on pages 33 and~34 in~\cite{bib:Trot-Book},
and the proof given there is complete and correct.
As it only takes a few lines and serves to set the stage for
a more subtle result to follow, we include an updated proof
of Lemma~\ref{lem:max-reverse-cps} in Section~\ref{sec:reversible}.

Recall that when $G$ is a graph, a subset $S$ of the vertex set of
$G$ is called an \textit{independent} set when no two vertices
in $S$ are adjacent in $G$.  The \textit{independence
number} of $G$, denoted $\alpha(G)$, is then defined as the
maximum cardinality of an independent set in $G$, and the
\textit{chromatic number} of $G$, denoted $\chi(G)$, is the
least positive integer $t$ for which the
vertex set of $G$ can be partitioned into $t$ independent
sets. If $G$ has $m$ vertices, then $\chi(G)\ge m/\alpha(G)$. This implies $\dim(S_n^k)\leq \chi(G_n^k)$, since the graph
$G_n^k$ of critical pairs of the crown $S_n^k$ has
$(n+k)(k+1)$ vertices and $\dim(S_n^k) = 2(n+k)/(k+2)$. As the opposite inequality holds for all graphs, Theorem~\ref{thm:chi=dim} is now seen to
be a corollary to Theorem~\ref{con:alpha(G)}.

A reversible set in $G_n^k$ is an independent set, but in general there are
independent sets which are not reversible.  These are sets of critical
pairs that contain one or more alternating cycles, but none of size~$2$.
Lemma~\ref{lem:max-reverse-cps} asserts that the maximum size of a reversible
set in $G_n^k$ is $(k+1)(k+2)/2$, leaving open the possibility that
there are independent, non-reversible sets which have size larger
than $(k+1)(k+2)/2$.  However, as is clear from our series of main theorems,
we show that this is not the case.

\medskip
\noindent
\textbf{Convention.}\quad  For the remainder of the paper, the
symbols $n$ and $k$ will only be used in reference to the crown
$S_n^k$ and the associated graph $G_n^k$ of critical pairs.
Accordingly, we always assume $n\ge3$ and $k\ge0$.

\subsection{Canonical Reversible Sets}

We say that a subset $X$ of minimal elements $A$ is \textit{contiguous} when the elements
of $X$ form a block of consecutive elements of $A$ with indices interpreted cyclically.   
For example, the set $X=\{a_8,a_1,a_9, a_2\}$ is
contiguous in $S_4^5$.  Vacuously, both $\emptyset$ and $A$ are contiguous.
Contiguous subsets of the maximal elements $B$ are defined analogously.

A sequence $\sigma=(x_1,x_2,\dots,x_r)$ of distinct elements of
$A$ will be called an $h$-\textit{contiguous sequence} when
$X_i=\{x_1,x_2,\dots,x_i\}$ is contiguous for all $i\in\{1,2,\dots,r\}$.
For example, in the crown $S_4^5$, the sequence 
$\sigma=(a_8,a_9,a_7,a_1,a_6,a_2)$ is $h$-contiguous.
The letter $h$ in this notation stands for ``hereditarily.''

When $\sigma=(a_1,a_2,\dots,a_r)$ is an $h$-contiguous sequence, we
let $T(\sigma)$ consist of all pairs $(a_1,b)\in \Inc(A,B)$ and, %for each $1\le i <r$, $T(\sigma)$ contains each $(a_{i+1},b)\in\Inc(A,B)$ for
%which $(a_i,b)\in T(\sigma)$.
for each $1\le i <r$, we include $(a_{i+1},b)$ in  $T(\sigma)$ provided $(a_{i+1},b)\in\Inc(A,B)$ and $(a_i,b)\in T(\sigma)$.

The inequality in Lemma~\ref{lem:max-reverse-cps} is easily seen to be tight,
as evidenced by the construction in the following lemma, which is 
implicit in the results of~\cite{bib:Trot-Snk}. 

\begin{lemma}\label{lem:canonical}
Let $(n,k)$ be a pair of integers with $n\ge3$ and $k\ge0$.
If $\sigma=(x_1,x_2,\dots,x_{k+1})$ is an $h$-contiguous sequence, 
then $T(\sigma)$ is reversible, $T(\sigma)$ is a maximal independent 
set in $G_n^k$, and $|T(\sigma)|=(k+1)(k+2)/2$.  
\end{lemma}

In the discussions to follow, we say that an independent
set $T$ in $G_n^k$ is a \textit{canonical reversible set} 
when there is an $h$-contiguous sequence $\sigma$ for which
$T=T(\sigma)$.

\begin{example}\label{exa:max-reverse}
To illustrate the preceding lemma and the connection between
$h$-contiguous sequences and canonical reversible sets, consider the crown 
$S_4^5$.  Then $\sigma=(a_8,a_9,a_7,a_1,a_6,a_2)$ is an $h$-contiguous 
sequence of length $6=k+1$.  The canonical reversible set $T=T(\sigma)$ 
associated with $\sigma$ is:
\begin{alignat*}{10}
T=  \{ &(a_8,b_8),&(a_8,b_9),&(a_8,b_1),&(a_8,b_2),&(a_8,b_3),&(a_8,b_4),\\
     &&(a_9,b_9),&(a_9,b_1),&(a_9,b_2),&(a_9,b_3),&(a_9,b_4),\\ 
     &&(a_7,b_9),&(a_7,b_1),&(a_7,b_2),&(a_7,b_3),\\ 
     &&&(a_1,b_1),&(a_1,b_2),&(a_1,b_3),\\ 
     &&&(a_6,b_1),&(a_6,b_2),\\
     &&&&\!(a_2,b_2)\phantom{,}&\!\}.
\end{alignat*}
The linear extension of $S_4^5$ represented by the following sequence, where the order is increasing left to right, reverses the
$6\cdot7/2=21$ pairs in $T$:
\begin{align*}
(a_3,a_4,a_5,\;\; b_2,a_2, \;\; b_1, a_6, \;\; b_3, a_1, \;\;b_9, a_7,\;\; b_4, a_9,\;\; b_8, a_8, \;\;b_5, b_6, b_7).
\end{align*}
\end{example}

Previously, we observed that the inequality $\dim(S_n^k)\ge
\lceil 2(n+k)/(k+2)\rceil$ follows from the fact that no linear extension
of $S_n^k$ can reverse more than $(k+1)(k+2)/2$ critical pairs.  
In~\cite{bib:Trot-Snk}, the reverse inequality is proved by
showing that the set of all critical pairs of $S_n^k$ can be covered
by $\lceil2(n+k)/(k+2)\rceil$ canonical reversible sets.  While a simple construction shows $\dim(S_n^k)\le 2\lceil (n+k)/(k+2)\rceil$, the improvement to $\dim(S_n^k)\le \lceil 2(n+k)/(k+2)\rceil$
takes a bit of work.

\subsection{Special Notation and Terminology for Crowns}\label{subsec:crowns}

As sets $A$ and $B$ are the minimal elements and maximal elements,
respectively, of $S_n^k$, the letter $a$ will always refers to a
minimal element, while $b$ is reserved for maximal elements.  In order
to avoid confusion over subscripts, we also use the letters $x$ and
$z$, sometimes with subscripts or primes, to denote elements of $A$,
while the letters $y$ and $w$ always represent elements of $B$.  On
the other hand, the letter $v$ is used to represent an element of
$S_n^k$ which may come from either $A$ or $B$. When $m\ge2$ and
$C=\{(x_\alpha,y_\alpha):1\le\alpha\le m\}$ is an alternating cycle,
we let $A(C)=\{x_1,x_2,\dots,x_m\}$ and $B(C)=\{y_1,y_2,\dots,y_m\}$.

Following the conventions of the presentation in~\cite{bib:Trot-Book},
we consider a circle in the Euclidean
plane with $n+k$ evenly spaced points on the circle labeled in 
clockwise order $(u_1,u_2,\dots,u_{n+k})$. 
We then imagine the elements of the crown $S_n^k$ placed on this circle with $a_i$ and $b_i$ both positioned on the point $u_i$. 

If $u_i, u_j,$ and $u_k$ are distinct points on the circle, then we write $u_i\prec u_j\prec u_k$ to signify that if we traverse the circle
in a clockwise direction, starting from $u_i$ and stopping the first time we encounter $u_k$, then somewhere in between, we saw $u_j$.\footnote{When we write $u_{j_1}\prec u_{j_2}\prec u_{j_3}\prec \ldots \prec u_{j_\ell}$, we intend that when traversing the circle in a clockwise direction, starting from $u_{j_1}$ and stopping the first time we encounter $u_{j_\ell}$, we visit $u_{j_2}, u_{j_3}, \ldots, u_{j_{\ell-1}}$ in that order.}
This notation can be used with inequalities which are not strict, so
the statement $u_i\prec u_j\preceq u_k$ implies that $u_i\neq u_j$ and $u_i\neq u_k$ while $u_j$ and $u_k$ could be the same. 

We extend this definition to include all elements of $S_n^k$. In particular, for distinct elements $v$, $v'$, and $v''$ of $S_n^k$,  we write $v \prec v' \prec v''$ if $u\prec u'\prec u''$ where $v$ is positioned at point $u$ on the circle, $v'$ at $u'$ and $v''$ at $u''$. 
Recall the convention that $a,x\in A$ and $y\in B$. So a statement like $x\prec y\preceq a$ includes the possibility that $y=b_i$ and $a=a_i$ for some $i$.

When $v$ and $v'$ are elements of $S_n^k$, $v$ located at position $u_i$ and $v'$ 
at $u_j$, then we say the \textit{size} of
$(v,v')$ is $j-i+1$, modulo $n+k$.  Note that this is simply the number of points visited on the circle (including the beginning and ending points) when we travel from $v$ to $v'$ moving in a 
clockwise manner. 
For example, in $S_4^5$, the size of $(a_7,b_1)$ is $4$ and the size of $(b_6,a_8)$ is $3$.

When $v_1,v_2\in A\cup B$, it is natural to say that
$(v_1,v_2)$ \textit{starts} at $v_1$ and \textit{ends} at $v_2$.
For $v_3,v_4\in A\cup B$, we say $(v_1,v_2)$ \textit{starts}
in $(v_3,v_4)$ if $v_3\preceq v_1\preceq v_4$ and 
\textit{ends} in $(v_3,v_4)$ when $v_3\preceq v_2\preceq v_4$. 

It is also natural to consider $\Inc(A,B)$ as an inclusion order, where
we say that $(a,b)$ is \textit{contained} in $(x,y)$ if $x\preceq a\preceq b\preceq y$.
Further, $(a,b)$ \textit{overlaps} $(x,y)$, if there is some
point $u$ on the circle so that $a\preceq u\preceq b$ and $x\preceq u\preceq y$.
If $(a,b)$ and $(x,y)$ do not overlap, we say they are \textit{disjoint}.
Note that $(a,b)$ and $(x,y)$ are adjacent in $G_n^k$ if and only if
they are disjoint and both $(b,x)$ and $(y,a)$ have size at most $n$.

Throughout the paper, we use the now standard notation $[m]$ to represent
the set $\{1,2,\dots,m\}$.  

\subsection{Natural Symmetries of the Crown $S_n^k$}\label{sec:auto}

The crown $S_n^k$ and the graph $G_n^k$ have two natural symmetries.
One of these captures the notion of rotation.
For each $j$, the map $\tau_j$ defined
by setting $\tau_j(a_i)=a_{i+j}$ and $\tau_j(b_i)=b_{i+j}$ is
both an automorphism of the crown $S_n^k$ and an automorphism
of the graph $G_n^k$.  

However, there is another natural symmetry that essentially
results from interchanging clockwise with counter-clockwise.
In particular, the map $\phi:A\cup B\rightarrow A \cup B$ defined by $\phi(a_i)=a_{-i}$ and
$\phi(b_j)=b_{k-j}$ is both an automorphism of the crown $S_n^k$ and an automorphism
of the graph $G_n^k$.  To see this, we observe that 
$(a_i,b_j)\in\Inc(A,B)$ if and only if $u_i \preceq u_j \preceq u_{i+k}$. 
Since $u_0=u_{n+k}$, we have the 
following string of equivalent inequalities:
\begin{align*}
u_i&\preceq  u_j \preceq u_{i+k},\\
u_0&\preceq u_{j-i}\preceq u_k,\\
u_0&\preceq u_{k-(j-i)} \preceq u_k,\\
u_{-i}& \preceq u_{k-j}\preceq u_{-i+k}.
\end{align*}

Moreover, if the size of the pair $(a,b)$ is $s$,
then the size of $(\phi(a),\phi(b))$ is $k+2-s$. As $\phi$ induces a bijection on $A\times B$, this implies $(a,b)\in \Inc(A,B)$ if and only if $(\phi(a), \phi(b))\in \Inc(A,B)$. Also note that 
if $(a,b),(x,y)\in\Inc(A,B)$, then $(a,b)$ is contained in 
$(x,y)$ if and only if $(\phi(x),\phi(y))$ is contained
in $(\phi(a),\phi(b))$.

When $S\subseteq\Inc(A,B)$, we let $\phi(S)$ denote the set
$\{(\phi(x),\phi(y)):(x,y)\in S\}$.  Clearly, (1)~$|S|=|\phi(S)|$,
(2)~$S$ is independent if and only $\phi(S)$ is independent and
(3)~$S$ is reversible if and only if $\phi(S)$ is reversible.

\subsection{Types of Independent Sets}

Recall that a reversible
set is always an independent set.  However, in general the inclusion
is strict as there may be independent sets that are not reversible.
For the rest of the paper, we follow the conventions that
(1)~the letter $T$ is used for independent sets known to be
canonical reversible sets; (2)~the letter $R$ is used for
independent sets known to be reversible (such sets may or may
not be canonical); and (3)~the letter $S$ is used for
independent sets when they are either non-reversible or the issue
of whether they are reversible has not been settled.
To avoid possible confusion with crowns, we denote independent
sets as $S$, $S'$ or $S''$, but we never use subscripts.

For a pair of integers $(n,k)$ with $n\geq 3$ and $k\geq 0$, we use the following notation concerning families
of subsets of the pairs in $G_n^k$.
\begin{enumerate}
\item $\mathbb{I}(n,k)$ is the family
all independent sets.
\item $\MI(n,k)$ is the family of all maximal independent
sets.
\item $\NI(n,k)$ is the family of all maximum size independent sets.
\item $\mathbb{R}(n,k)$ is the family
of all reversible sets.
\item $\MR(n,k)$ is the family of
all maximal reversible sets.
\item $\mathbb{MSR}(n,k)$ is the family of all maximum size reversible sets.
\item $\INR(n,k)$ is the family of all independent, non-reversible sets.
\item $\MINR(n,k)$ is the family of all maximal independent, non-reversible sets.
\item $\NINR(n,k)$ is the family of all maximum size independent, non-reversible sets.
\end{enumerate}
%When $a\in A$, we let $I(a)$ denote the set of all $b\in B$ for which
%$(a,b)\in\Inc(A,B)$.  For each $b\in B$, $I(b)$ is the set of all
%$a\in A$ with $(a,b)\in\Inc(A,B)$.  Of course, if $a=a_i$, then
%$I(a_i)=\{b_i,b_{i+1},\dots,b_{i+k}\}$ and if $b=b_j$, then
%$I(b)=\{a_j,a_{j-1},\dots,a_{j-k}\}$.
When the pair $(n,k)$ has been fixed, we abbreviate the names of the above sets by leaving off the $(n,k)$. For example, $\mathbb{I}(n,k)$ is abbreviated $\mathbb{I}$. 

Fix a pair $(n,k)$. 
 When $S\in\mathbb{I}$, let $A(S)$ consist of all elements
$a\in A$ for which there is some $b\in B$ with $(a,b)\in S$.
The set $B(S)$ is defined analogously.
 For each $a\in A$,
let $B(a,S)$ denote the set of all
$b\in B$ for which $(a,b)\in S$. 
For $b\in B$, the set $A(b,S)$ is defined analogously.
When $S=\Inc(A,B)$, set $I(a):= B(a,S)$ and $I(b):=A(b,S)$. 
Of course, $I(a_i)=\{b_i,b_{i+1},\dots,b_{i+k}\}$ and $I(b_j)=\{a_j,a_{j-1},\dots,a_{j-k}\}$. Further, for any $S\in\mathbb{I}$, $B(a,S)\subseteq I(a)$, $A(b,S)\subseteq I(b)$. If $a\not\in A(S)$, then $B(a,S)=\emptyset$. Likewise $A(b,S)=\emptyset$ when $b\not\in B(S)$.

\section{Reversible Sets}\label{sec:reversible}

This section is devoted to the study of reversible sets and 
includes the proof of Theorem~\ref{thm:main-1}, the first
of our three main theorems.  As there are three statements in this
theorem, we will remind readers of the wording of the individual
statements at the appropriate moment in the argument.

Fix a pair $(n,k)$.
Let $R\in\MR$ and let $L$ be a linear extension of $S_n^k$ which
reverses all pairs in $R$.  Scan $L$ from bottom to top
and note that $L$ is a linear order on $A\cup B$ consisting
of blocks of elements of $A$ interspersed with blocks
of elements of $B$.  The bottom block is $A-A(R)$ while the top block
is $B-B(R)$. Accordingly, $L$ has the following \textit{block-structure}
form, where the order of the blocks is increasing left to right:
\[
(A_{s+1}, B_s, A_s, B_{s-1}, A_{s-1}, B_{s-2}, \dots, 
A_3, B_2, A_2, B_1, A_1, B_0).
\]
In particular, $A_{s+1} = A-A(R)$ and $B_0=B-B(R)$.

It is easy to see that any linear extension $L$ that reverses
$R$ has this form; furthermore, the only allowable variation is
the linear order imposed on elements within a block as $R$ is maximal.  The set $R$ determines the block structure triple $(s,\cgF, \cgG)$: the integer $s$ and the 
two set partitions $\cgF=\{A_1,A_2,\dots,A_{s+1}\}$ and
$\cgG=\{B_0,B_1,\dots,B_s\}$ of $A$ and $B$, respectively. 

We note that since $L$ is a linear extension of $P$, it
satisfies the following condition.

\medskip
\noindent
\textit{Admissibility Condition.}\quad If $x\in A_i$, $y\in B_j$, and
$x<y$ in $S_n^k$, then $0\le j < i\le s+1$.

\medskip
\noindent Since $R$ is maximal, $L$ also satisfies the following
condition.

\medskip
\noindent
\textit{Maximality Condition.}\quad For each $i\in\{0,1,\dots,s\}$ if 
$x\in A_{i+1}$ and $y\in B_i$, then $x<y$ in $S_n^k$.

\medskip
\noindent To see that the Maximality Condition is satisfied, suppose there
is some $i$ and a pair $(x,y)\in\Inc(A,B)$ 
with $x\in A_{i+1}$ and $y\in B_i$.  Then a linear extension with
the following block form reverses all pairs in $R\cup\{(x,y)\}$,
contradicting the assumption that $R\in\MR$:
\[
(A_{s+1}, B_s, A_s,  \dots, 
B_{i+1},A_{i+1}-\{x\},\{y\},\{x\},B_i-\{y\}, A_i,
\dots, B_1, A_1, B_0).
\]

We pause here to prove the first statement of Theorem~\ref{thm:main-1}: If $R\in\MR$, then $R\in\MI$. Let $(x,y)\in\Inc(A,B)-R$ be arbitrary.
We argue that there exists $(z,w)\in R$ such that $(z,w)$ is adjacent to $(x,y)$ in $G_n^{k}$.  Since $R$ satisfies the maximality condition, there are integers
$i,j$ with $0\le i\le j-2\le s-1$ so that
$x\in A_j$ and $y\in B_i$. Let $w$ be any element of $B_{j-1}$ and let $z$ be any element
of $A_{i+1}$.  It follows that the incomparable pair $(z,w)$ is reversed and thus is in $R$.  By the maximality condition, $x< w$ and $z< y$. Therefore, 
$(x,y)$ and $(z,w)$ are adjacent in $G_n^k$. This completes the proof.

In some sense, we now have characterized
the sets in $\MR$, as they are exactly the sets $R$
associated with a block structure triple $(s,\cgF,\cgG)$ satisfying the
admissibility and maximality conditions.
However, in the work to follow, we need to know the largest
two sizes these sets can have. 

We note that when $R$ is reversible, if
$a,a'\in A(R)$, then one of $B(a,R)$ and $B(a',R)$ is
a subset of the other.  Also, it may happen that
$B(a,R)=B(a',R)$. Analogous remarks apply when $b,b'\in B$.

When $R\in\MR$, it is easy to see that  
$|A_{s+1}|=|B_0|=n-1$, $|A_1|=|B_s|=1$, and $|A(R)|=|B(R)|=k+1$.
Furthermore, $R$ is a canonical reversible set if and only if
$s=k+2$.  In this case, $|A_i|=|B_{i+1}|=1$ for all $i\in[k+1]$.

Again, let $R\in\MR$.  A labeling $\{x_1,x_2,\dots,x_{k+1}\}$ of $A(R)$
is called a \textit{consistent} labeling 
if $\alpha <\beta$ whenever $B(x_\beta, R)$ is a subset of $B(x_\alpha,R)$. One such labeling can be obtained from any ordering of $A(R)$ with block structure $(A_1, A_2, \dots, A_s, A_{s+1})$ ordered left to right.
Here is an updated version
of the proposition at the heart of the proof
of Lemma~\ref{lem:max-reverse-cps}, as given in~\cite{bib:Trot-Book}.

\begin{proposition}\label{pro:correct}
Let $R\in\MR$ and let $\{x_1,x_2,\dots,x_{k+1} \}$ be
a consistent labeling of $A(R)$. Then for each
$i\in[k+1]$, $|B(x_i,R)|\le k+2-i$.
\end{proposition}

%\begin{proof}
%The inequality holds (and is tight) when $i=1$, 
%so we may assume that $1<i\le k+1$. Let $b_s$ be any element of $B(x_i,R)$.  
%We note that $(x_j,b_s)\in R$ for every $j=1,2,\dots,i$.  Accordingly, we
%can relabel the elements of $\{x_1,x_2,\dots,x_i\}$ as
%$\{z_1,z_2,\dots,z_i\}$ so that $a_{s-k}\preceq z_1\prec z_2\prec\dots\prec z_i\preceq a_s$.  
%Now, we note that $B(x_i,R)$ is a subset of the $k+1$-element subset
%$I(z_1)$.  For each $j=2,3,\dots,i$, let
%$z_j=a_{\beta_j}$.  Then $b_{\beta_j-1}\in I(z_1)$.  However,
%$z_j=a_{\beta_j}<b_{\beta_j-1}$ in $S_n^k$, which implies that $b_{\beta_j-1}
%\not\in B(x_i,R)$. As the elements in $\{b_{\beta_j-1}:2\le j\le i\}$ are
%distinct, we conclude that $|B(x_i,R)|\le (k+1)-(i-1)=k+2-i$. 
%\end{proof}

\begin{proof}
The inequality holds (and is tight) when $i=1$, 
so we may assume that $1<i\le k+1$. 
%We will find $i-1$ elements in $I(x_i)$ that are not reversed in any linear extension that reverses $R$.

Let $b_{t}$ be any element of $B(x_i,R)$.  
Then $\{x_1, x_2, \ldots, x_i\}\subseteq A(b_t, R)$ by definition of  a consistent labeling. 
Therefore, we
can relabel the elements of $\{x_1,x_2,\dots,x_i\}$ as
$\{z_1,z_2,\dots,z_i\}$ so that $a_{t-k}\preceq z_1\prec z_2\prec\dots\prec z_i\preceq b_{t}$.
Again by the definition of a consistent labeling, $B(x_i,R) \subseteq B(z_1,R)\subseteq I(z_1)$. Now we find $i-1$ elements in $I(z_1)$ which are not in $B(x_i,R)$ to bound the size of $B(x_i, R)$.

For each $j\in\{2,3,\dots,i\}$, if $z_j=a_{\beta_j}$, then $b_{\beta_j-1}\in I(z_1)$ because $a_{t-k} \preceq z_1 \prec a_{\beta_j} \preceq b_t$.  However,
$z_j=a_{\beta_j}<b_{\beta_j-1}$ in $S_n^k$, which implies that $b_{\beta_j-1}
\not\in B(x_j,R)$. Because $B(x_i,R)\subseteq B(x_j,R)$, we have $b_{\beta_j-1}\not\in B(x_i,R)$. Since $b_{\beta_j-1}\in I(z_1) - B(x_i,R)$ for each $j\in \{2,3, \dots, i\}$ and  $B(x_i,R) \subseteq I(z_1)$, we have $|B(x_i,R)| \le (k+1)-(i-1)=k+2-i$. 
\end{proof}

Lemma~\ref{lem:max-reverse-cps} follows immediately from
the preceding proposition since we have
\[
|R|=\sum_{i=1}^{k+1}|B(x_i,R)|\le \sum_{i=1}^{k+1}k+2-i=\frac{(k+1)(k+2)}{2}.
\]

Although the inequality in Proposition~\ref{pro:correct} is tight for
canonical reversible sets, if $R\in\MR$ and $R$ is not canonical,
it is not tight.   For such sets, we have the following stronger result.

\begin{lemma}\label{lem:reversible-NC}
Let $R$ be any set in $\MR$ which is not a canonical reversible set, and
let $\{x_1,x_2,\dots,x_{k+1}\}$ be a consistent labeling of $A(R)$.
Choose $i$ to be the least positive integer for which
$|B(x_i,R)|\neq k+2-i$. Then the following hold:
\begin{enumerate}
\item $i\geq 2$ and $i \leq k-n+2$;
\item $|B(x_i,R)|=k+4-i-n$; and
\item For every $j\in\{i+1,i+2,\dots,k+1\}$, $|B(x_j,R)|\le k+3-i-n$.
\end{enumerate}
\end{lemma}

\begin{proof}
We have already noted $|B(x_1,R)|=k+1$, so
$i\ge 2$.  The assumption that $|B(x_j,R)|=k+2-j$ for each $j\in[i-1]$
implies that the sequence $(x_1,x_2,\dots,x_{i-1})$ is $h$-contiguous.
The assumption $|B(x_i,R)|\neq k+2-i$ implies $\{x_1,x_2,\dots,x_i\}$ is
not contiguous.  After relabeling,
we may assume $\{x_1,x_2,\dots,x_{i-1}\}=\{a_1,a_2,\dots,
a_{i-1}\}$, and further $x_i\not\in\{a_{n+k},a_{i}\}$. By the definition of a consistent labeling, $B(x_{i-1}, R) \subseteq B(a_j,R)$ for each $j\in [i-1]$. Therefore $B(x_{i-1},R)\subseteq B(a_1,R) \cap B(a_{i-1},R) \subseteq \{b_1, \ldots, b_{k+1}\} \cap \{b_{i-1}, \ldots, b_{i-1+k}\} = \{b_{i-1}, \ldots, b_{k+1}\}$. Since $|B(x_{i-1},R)| = k-i+3$, we conclude that $B(x_{i-1},R)=\{b_{i-1},b_i,\dots,b_{k+1}\}$.

Our aim is to show $\{b_{i-1},b_{k+1}\}\subseteq I(x_i).$ By the consistent labeling, $B(x_i,R)\subseteq B(x_{i-1},R) = \{b_{i-1},b_i,\dots,b_{k+1}\}$ with $B(x_i,R)\neq\emptyset$ as $x_i\in A(R)$. 

For contradiction, first suppose $x_i<b_{i-1}$ and $x_i<b_{k+1}$. Therefore  $I(x_i)\subseteq\{b_i,\ldots,b_k\}$. Thus $|I(x_i)|<k+1,$ a contradiction. So at most one of $x_i<b_{k+1}$ and $x_i<b_{i-1}$ is possible. 

Toward a contradiction, suppose $x_i<b_{i-1}$ and $b_{k+1}\in I(x_i).$ Then $x_i<b_i,$ since $x_i \neq a_i$ and $|B(x_i,R)|<k+2-i$. In the block structure, $x_i\in A_i$ and $\{b_{i-1},b_i\}\subseteq B_{i-1}.$ Then $(a_i,b_i)\not\in R.$ But $R\cup(a_i,b_i)$ is reversible because we can replace the blocks $(A_i,B_{i-1})$ with $(A_i,\{b_i\},\{a_i\},B_{i-1}-\{b_i\}),$ which contradicts the maximality of $R.$ It can also be shown that $x_i<b_{k+1}$ and $b_{i-1}\in I(x_i)$  cannot occur by following a similar argument. In this case, we contradict the maximality of $R$ by showing that $R\cup(a_{n+k},b_{k})$ is reversible.

We now know that $\{b_{i-1},b_{k+1}\}\subseteq I(x_i).$ Since $B(x_i,R)\subsetneq\{b_{i-1},b_i,\ldots,b_{k+1}\}$, it follows that the set of $n-1$ elements $x_i$ is comparable to a subset of $\{b_i,\ldots,b_{k}\}$. Therefore $n-1\leq k+1-i,$ which is equivalent to the first statement of the lemma. In particular, $|B(x_i,R)| = (k+3-i) - (n-1)$, the second statement of the lemma. 

For the third statement, suppose $x_i=a_m,$ fix $j\in\{i+1,i+2,\ldots,k+1\}$, and suppose $x_j=a_p.$ Note that $B(x_j,R)\subseteq B(x_i,R)$. In each of three different cases based on the value of $p$, we show that $B(x_i,R) - B(x_j,R) \neq \emptyset$.  If $i\leq p \leq m-n$ or $m+1\leq p\leq k+1$, then $b_{p-1}\in B(x_i,R)-B(x_j,R).$  If $m-n+1\leq p \leq m-1$  then $b_{m-n}\in B(x_i,R)-B(x_j,R)$. If $k+2\leq p \leq n+k$, then $b_{k+1}\in B(x_i,R)-B(x_j,R).$ So in any case, $B(x_i,R)-B(x_j,R)\neq\emptyset$ and therefore $|B(x_j,R)|\leq k+3-i-n.$
\end{proof}

The first statement of Lemma~\ref{lem:reversible-NC} yields the proof of the
second statement in Theorem~\ref{thm:main-1}, i.e., if $n>k$ and
$S\in\MR$, then $S$ is a canonical reversible set.  %In particular, because $i\geq 2$, this inequality $i \leq k-n+2$ implies $n\le k$.

Lemma~\ref{lem:reversible-NC} also supplies the proof of the inequality in the third statement
of Theorem~\ref{thm:main-1}: If $R\in\MR$ and
$R$ is not a canonical reversible set, then $|R|\le (k+1)(k+2)/2-
n(n-1)/2+1$.  To see this, let $\{x_1,x_2,\dots,
x_{k+1}\}$ be a consistent labeling of $A(R)$.  As Lemma~\ref{lem:reversible-NC},
let $i$ be the least integer for which $|B(x_i,R)|\neq k+2-i$. The lemma then supplies the exact size of $B(x_j,R)$ when $1\le j\le i$.
When $i+1\le j\le k+1$, the lemma only tells us that $|B(x_j,R)|\le
k+3-i-n$.  However, from Proposition~\ref{pro:correct}, we also
have the inequality $|B(x_j,R)|\le k+2-j$.
It follows that
\begin{align*}
|R|\le&\sum_{j=1}^{i-1}(k+2-j)
      +(k+4-i-n)+(n-1)(k+3-i-n)+\sum_{j=i+n}^{k+1}(k+2-j)\\
    =&\frac{(k+1)(k+2)}{2}-\frac{n(n-1)}{2}+1.
\end{align*}
This completes the proof of Theorem~\ref{thm:main-1}.

The next example shows that the inequality in the third
statement of Theorem~\ref{thm:main-1} is tight. 

\begin{example}\label{exa:non-consecutive}
Let $(n,k)$ be a pair with $3\le n\le k$. We start with the
canonical reversible set $T=T(\sigma)$ consisting of all pairs associated with the $h$-contiguous sequence $\sigma=(a_1, a_2, \ldots, a_{k+1})$.
Let $i$ be any integer with $1\le i\le k+1-n$.
Then $u_1\preceq u_i \prec u_{i+n}\preceq u_{k+1}$ where $(a_{i+n},b_i)$ belongs to $\Inc(A,B)$ but not
to $T$.  

A pair $(a,b)\in T$ is adjacent to
$(a_{i+n},b_i)$ in $G_n^k$ if and only if
$u_i\prec a\preceq b\prec u_{i+n}$, so the pair $(a_{i+n},b_i)$ has
$n(n-1)/2$ neighbors in $T$.  If we remove these neighbors
from $T$ and then add $(a_{i+n},b_i)$, we have an independent
set $R$ of size $(k+1)(k+2)/2-n(n-1)/2+1$.  This set
is easily seen to belong to $\MR$.
\end{example}

\section{Independent, Non-reversible Sets: Part 1}\label{sec:part-1}

In some sense, independent, reversible sets are relatively
simple objects.  To our taste, independent, non-reversible sets
are much more interesting, and
this is the first of three sections devoted to such sets.

For a pair $(n,k)$, recall that $\INR(n,k)$ denotes the family of
all independent, non-reversible sets in $G_n^k$.
In this section, we prove Lemma~\ref{lem:INR-exists}: $\INR(n,k)$  is non-empty
if and only if $n\le 2k$.  
%Also, recall that $\MINR(n,k)$ is the
%family of all maximal independent, non-reversible sets and
%$\NINR(n,k)$ is the family of all maximum size independent, non-reversible
%sets. 

We begin by reminding readers that independent, non-reversible
sets in $G_n^k$ contain alternating cycles. In fact, they contain
strict alternating cycles.  However, since they are independent,
any alternating cycle they contain has size at least~$3$.
The following straightforward proposition is essentially 
a restatement of the basic properties of a strict alternating cycle.

\begin{proposition}\label{pro:AC}
Let $S$ be a non-reversible set in $G_n^k$. Then let
$m\ge2$ and let $C=\{(x_\alpha,y_\alpha):\alpha\in [m]\}$ be an
alternating cycle contained in $S$. Then for each $\alpha\in [m]$, we have
$x_\alpha\preceq y_\alpha\prec x_{\alpha+1}$.  Furthermore, if $S$ is strict, 
there is no $v\in A(C)\cup B(C)$ with $y_\alpha\prec v\prec x_{\alpha+1}$.
\end{proposition}

As a result, when $m\ge2$ and $C=\{(x_\alpha,y_\alpha):\alpha\in [m]\}$
is a strict alternating cycle, there are $(m-1)!$ different orders in which the elements of $A(C)\cup B(C)$ can appear on the 
circle.  For example, when $m=7$, one of these possibilities is:
\[
x_1\preceq  y_3\prec x_4\preceq  y_6\prec x_7\preceq  y_4\prec x_5\preceq  y_2\prec x_3\preceq  y_1\prec x_2\preceq  y_5\prec x_6\preceq  y_7.
\]

Furthermore, for every such arrangement, there
is a strict alternating cycle of this form in $\Inc(A,B)$, when
$k$ is sufficiently large relative to $n$.  

\begin{lemma}\label{lem:dots}
Let $m\ge2$ and let $C=\{(x_\alpha,y_\alpha): \alpha\in [m]\}$ be a
strict alternating cycle in $\Inc(A,B)$.  Then $mn\le 2(n+k)$.
\end{lemma}
 
\begin{proof}
The conclusion of the lemma holds trivially if $m=2$, so we assume
$m\ge3$.  For each $\alpha\in[m]$, let $D_\alpha$ be the down set of $y_\alpha$ in $S_n^k$. If $y_\alpha=b_j$, then
$D_\alpha=\{a_{j-1},a_{j-2},\dots,a_{j-n+1}\}$ which is a set of size $n-1$. 

Since $C$ is a strict alternating cycle, it follows that
for each $a\in A(C)$, there is exactly one $\alpha\in[m]$ with $a\in D_\alpha$. Provided each element $a\in A-A(C)$ belongs to at most two of the sets in $\{D_\alpha:\alpha\in [m]\}$, then $m(n-1) = \sum_{\alpha\in[m]} |D_\alpha|\le 2(n+k-m)+m$ which implies $mn\le 2(n+k)$ as claimed.

It remains to see that each $a\in A-A(C)$ belongs to at most two downsets in $\{D_\alpha: \alpha\in[m]\}$. Suppose to the contrary that there is some $a\in A-A(C)$ such
that $a\in D_\alpha \cap D_\beta \cap D_\gamma$ for three
distinct elements $\alpha,\beta,\gamma\in[m]$.  Without loss
of generality, we may assume that $y_\alpha\prec y_\beta\prec y_\gamma \prec a$. 
Since $D_\alpha$, $D_\beta$ and $D_\gamma$ have the same size, $D_\beta \subseteq D_\alpha \cup D_\gamma$. Recall $x_{\beta+1} < y_\beta$ in the cycle $C$, so $x_{\beta+1}\in D_\beta$. Therefore, $x_{\beta+1}\in D_\alpha$ or $x_{\beta+1}\in D_\gamma$, which implies $x_{\beta+1} < y_\alpha$ or $x_{\beta+1} < y_\gamma$, contradicting the assumption that $C$ was a strict alternating cycle. So, each element $a\in A-A(C)$ belongs to at most two of the sets in $\{D_\alpha:\alpha\in [m]\}$, as needed to complete the proof.
\end{proof}

We are now ready to prove the first half of Lemma~\ref{lem:INR-exists}: If $\INR\neq\emptyset$, then $n\le 2k$.  
Let $S\in\INR$.  Then $S$ contains a strict alternating cycle $C$ of size $m$ for some $m\geq 3$.
Lemma~\ref{lem:dots} implies $3n\le mn\le 2(n+k)$, and
therefore $n\le 2k$.

For the converse, we simply show that when $n\le 2k$, there is a strict
alternating cycle of size~$3$ in $\Inc(A,B)$. Two suitable examples are provided here.

\begin{example}\label{exa:k ge n}
If $n\le k$, then
\[
C=\{(a_1,b_1),(a_2,b_{k+1}), (a_{k+2},b_{k+2})\}
\]
is a strict alternating cycle of size~$3$ in $G_n^k$.
Accordingly, $C\in \INR(n,k)$.
\end{example}

\begin{example}\label{exa:k < n}
If $k < n\le 2k$, then
\[
C=\{(a_1,b_{2k+1-n}),(a_{k+1},b_{k+1}), (a_{2k+1},b_{2k+1})\}
\]
is a strict alternating cycle of size~$3$ in $G_n^k$.
Accordingly, $C\in \INR(n,k)$.
\end{example}

In time, it will become clear why these last two examples
are presented in terms of the separate ranges: (i)~$n\le k$ and
(ii)~$k<n\le 2k$. We also alert readers that the 
strict alternating cycles in these examples will resurface later in 
this paper.

Even though there can be arbitrarily complex strict alternating cycles
in $\Inc(A,B)$ when $k$ is sufficiently large relative to $n$, the
following lemma asserts that there is always one of small size in
a maximal independent, non-reversible set.

\begin{lemma}\label{lem:small-sac}
If $S\in\MINR(n,k)$, then $S$ contains a strict alternating cycle of size~$3$.
\end{lemma}

\begin{proof}
As $S$ is non-reversible, we can choose a strict alternating cycle $C=\{(x_\alpha,y_\alpha):\alpha\in[m]\}$ in $S$ of smallest size. Of
course, $m\ge3$ since $S$ is independent.
Toward a contradiction, suppose $m\geq 4$. Then
$(x_1,y_2)\in\Inc(A,B)$.  If $(x_1,y_2)\in S$, then
we may delete $(x_1,y_1)$ and $(x_2,y_2)$ from $C$ and add
$(x_1,y_2)$ to obtain a strict alternating cycle of size $m-1$ contained
in $S$.  The contradiction shows that $(x_1,y_2)\not\in S$.

Since $S$ is maximal, it follows that there is some $(a,b)\in S$ with
$(a,b)$ adjacent to $(x_1,y_2)$ in $G_n^k$.  Therefore, $a<y_2$
and $x_1<b$.  Thus $C'=\{(x_1,y_1),(x_2,y_2),(a,b)\}$
is a strict alternating cycle in $S$ of size 3, another contradiction.
\end{proof}

As this detail is essential to
future arguments, we note that when $C=\{(x_\alpha,y_\alpha):
\alpha\in[3]\}$ is a strict alternating cycle of size~$3$,
there are only two different ways the points of $A(C)\cup B(C)$ 
can appear:

\begin{enumerate}
\item $x_1 \preceq  y_1 \prec  x_2 \preceq  y_2 \prec  x_3 \preceq  y_3$ or 
\item $x_1 \preceq  y_2 \prec  x_3 \preceq  y_1 \prec  x_2 \preceq  y_3$.
\end{enumerate}

Considering our conventions about placing points
of $S_n^k$ on a circle, it is natural to say that $C$ satisfies the \textit{Disjoint Property} when
the first of these two orders holds. Similarly, $C$ satisfies
the \textit{Overlap Property} when the second order holds.
Clearly, if $C$ is a strict alternating cycle of size~$3$, then
$C$ satisfies the Disjoint Property if and only if its image under the automorphism $\phi$ (see Section~\ref{sec:auto})
satisfies the Overlap Property.

Strict alternating cycles of size~$3$ play a vital role in the
arguments to follow.  Accordingly, 
let $\CDP$ denote the family of all strict alternating
cycles of size~$3$ in $\Inc(A,B)$ that satisfy the Disjoint Property while 
$\mathbb{C}_{O3}$ consists of those with the Overlap Property.  

For a fixed pair $(n,k)$ with $\INR\neq\emptyset$, we
use the notation $\MINR_{D3}$ for the family of sets in $\MINR$ which
contain a strict alternating cycle $C$ from $\CDP$.
The family $\NINR_{D3}$ is
defined similarly.  Analogously, define the families
$\MINR_{O3}$ and $\NINR_{O3}$ in terms of the Overlap Property.
In view of our remarks about the automorphism $\phi$,
we state the following elementary proposition for
emphasis.

\begin{proposition}\label{pro:D3-O3}
Let $(n,k)$ be a pair for which $\INR\neq\emptyset$.
Then
\begin{enumerate}
\item $\MINR_{O3}=\{\phi(S):S\in\MINR_{D3}\}$ and $\MINR=\MINR_{D3}\cup\MINR_{O3}$;
\item $\NINR_{O3}=\{\phi(S):S\in\NINR_{D3}\}$ and $\NINR=\NINR_{D3}\cup\NINR_{O3}$.
\end{enumerate}
\end{proposition}

In carrying out the research for this paper, we found it more intuitive
to work with the families $\MINR_{D3}$ and $\NINR_{D3}$, but
as reflected by Proposition~\ref{pro:D3-O3}, if we understand these
families, we really understand $\MINR$ and $\NINR$.

\subsection{The Contraction and Expansion Lemmas}
Fix a pair $(n,k)$ and recall $\mathbb{I}$ is the family of all independent sets in $G_n^k$.
Let $S\in\mathbb{I}$ and $i\in[n+k]$. An ordered
pair $((a,b),(x,y))$ of critical pairs belonging to $S$ is called a \textit{contraction blocking pair at $i$} if
$a=a_i$, $y=b_i$ and $x<b$ in $S_n^k$.
To explain this terminology, we observe that both $(a_{i+1},b)$ and 
$(x,b_{i-1})$ are in $\Inc(A,B)$, and they represent ``contractions'' of the
pairs $(a,b)$ and $(x,y)$, respectively.  However,
$S$ does not contain either as $(a_{i+1},b)$ is adjacent to
$(x,y)$ and $(x,b_{i-1})$
is adjacent to $(a,b)$.

We then let $\FCBP(i,S)$, for ``first in a contraction blocking pair" at $i$, denote the set of all 
$(a,b)\in S$ for which there is some 
$(x,y)\in S$ such that $((a,b),(x,y))$ is a contraction
blocking pair at~$i$. Also, we let $\LCBP(i,S)$, for ``last in a contraction blocking pair'' at $i$, denote the set of
all $(x,y)\in S$ for which there
is some $(a,b)\in S$ such that
$((a,b), (x,y))$ is a contraction blocking pair at~$i$.
 It follows from the definition that
$\FCBP(i,S)\neq\emptyset$ if and only if $\LCBP(i,S)\neq\emptyset$.

We need two more definitions that transform an independent set via a contraction. In particular, the notation $\DFCL$ is an
abbreviation for ``delete first, contract last,'' while
$\DLCF$ is an abbreviation for ``delete last, contract first.'' Set
\begin{align*}
\DFCL(i,S)&:=(S-\FCBP(i,S))\cup \{(x,b_{i-1}):(x,y)\in \LCBP(i,S)\},\\
\DLCF(i,S)&:=(S-\LCBP(i,S))\cup \{(a_{i+1},b):(a,b)\in \FCBP(i,S)\}.
\end{align*}
%In discussions to follow, the next lemma will be referred to
%as the ``contraction lemma.'' 

\begin{lemma}[Contraction Lemma]\label{lem:contraction-block} 
Fix $(n,k)$. Let $S\in\mathbb{I}$ and $i\in[n+k]$.
\begin{enumerate}
\item $\DFCL(i,S)\in\mathbb{I}$.
\item $\DLCF(i,S)\in\mathbb{I}$.
\item $|\DFCL(i,S)|+|\DLCF(i,S)|=2|S|$.
\end{enumerate}
\end{lemma}

\begin{proof}
The conclusions of the lemma hold trivially if $\FCBP(i,S)$ and
$\LCBP(i,S)$ are empty.  Now suppose that both of these sets are
non-empty.  We show that $\DFCL(i,S)\in\mathbb{I}$, noting that the argument 
for showing $\DLCF(i,S)\in\mathbb{I}$ is symmetric. 

 The subset $S\cap 
\DFCL(i,S)$ is independent because $S$ is independent.  Also, no two
elements of $\DFCL(i,S)-S$ can be adjacent since they all end at
$b_{i-1}$.
Toward a contradiction, suppose there is some $(z,w)\in S\cap\DFCL(i,S)$ and some $(x,b_{i-1})\in \DFCL(i,S)-S$ with
$(z,w)$ adjacent to $(x,b_{i-1})$.  This requires $x<w$ 
and $z<b_{i-1}$ in $S_n^k$.  Since $(x,b_i)\in S$, we know that
$(z,w)$ is not adjacent to $(x,b_i)$. This forces 
$z=a_i$.  As $x<w$, then $((z,w),(x,b_i))$ is a contraction 
blocking pair at $i$ so that $(z,w)\in\FCBP(i,S)$ and $(z,w)\not\in \DFCL(i,S)$, a contradiction. Thus $\DFCL(i,S)\in \mathbb{I}$.

From their definitions, it follows that
\begin{align*}
|\DFCL(i,S)|=&|S|-|\FCBP(i,S)|+|\LCBP(i,S)|\quad\text{and}\\
|\DLCF(i,S)|=&|S|-|\LCBP(i,S)|+|\FCBP(i,S)|.
\end{align*}
Adding these two identities yields $|\DFCL(i,S)|+|\DLCF(i,S)|=2|S|$.
\end{proof}

Similar to contraction blocking pairs, we also have expansion blocking pairs. An ordered pair $((a,b),(x,y))$ of critical pairs belonging to 
$S$ is called an \textit{expansion blocking pair at 
$i$} when $a=a_i$, $y=b_{i+k}$, and $x<b$ in $S_n^k$.  Note that 
$S$ does not contain either $(a_{i-1},b)$ or $(x,b_{i+k+1})$.

Let $\FEBP(i,S)$ denote the set of all $(a,b)\in S$
for which there is some $(x,y)\in S$
such that $((a,b),(x,y))$ is an expansion
blocking pair at $i$. Also, let $\LEBP(i,S)$ denote the set of all 
$(x,y)\in S$ for which there is some 
$(a,b)$ in $S$ such that $((a,b),(x,y))$ is an expansion 
blocking pair at~$i$.    As before, $\FEBP(i,S)\neq\emptyset$ 
if and only if $\LEBP(i,S)\neq\emptyset$. 

As with contractions, we define two sets which transform an independent set via an expansion: ``delete first, expand last'' and ``delete last, expand first.'' Set
\begin{align*}
\DFEL(i,S)&:=(S-\FEBP(i,S))\cup \{(x,b_{i+k+1}):(x,b_{i+k})\in 
\LEBP(i,S)\},\\
\DLEF(i,S)&:=(S-\LEBP(i,S))\cup \{(a_{i-1},b):
(a_i,b)\in \FEBP(i,S)\}.
\end{align*}

The proof of the following lemma is essentially the same as for
Lemma~\ref{lem:contraction-block}.

\begin{lemma}[Expansion Lemma]\label{lem:expansion-block}
Fix $(n,k)$. Let $S\in\mathbb{I}(n,k)$ and $i\in[n+k]$.

\begin{enumerate}
\item $\DFEL(i,S)\in\mathbb{I}$.
\item  $\DLEF(i,S)\in\mathbb{I}$.
\item $|\DFEL(i,S)|+|\DLEF(i,S)|=2|S|$.
\end{enumerate}
\end{lemma}

As this
detail will be important later, we note that if
$|\FCBP(i,S)| \neq |\LCBP(i,S)|$, then one of 
$|\DFCL(i,S)|$ and $|\DLCF(i,S)|$ is larger than $|S|$, while the
other is smaller. An analogous statement holds for expansions.
Furthermore, if $S\in\NI$, then
 $\DFCL(i,S)$, $\DLCF(i,S)$, $\DFEL(i,S)$, and $\DLEF(i,S)$ are in $\NI$.  

While the transformations in the contraction and expansion
lemmas preserve membership in $\NI$, they may not preserve membership in $\MI$, as the following
two examples show.
Also, they may transform a non-reversible set into a reversible set.
We encourage the reader to work through the claims of these two
examples in preparation for the arguments to follow in the
next two sections.

\begin{example}\label{exa:k < n-extreme}
Let $(n,k)$ be a pair with $k<n\le 2k$. Recall the following
strict alternating cycle, first introduced in Example~\ref{exa:k ge n}:
\[
C=\{(a_1,b_{2k+1-n}),(a_{k+1},b_{k+1}), (a_{2k+1},b_{2k+1})\}.
\]
Set \[S:=\{(x,y): a \preceq x \preceq y \preceq b \text{ for some } (a,b)\in C\}.\]
Then the following statements hold.

\begin{enumerate}
\item $S\in\MINR(n,k)$ and $|S|=2+(2k+2-n)(2k+1-n)/2$.
\item $\FEBP(1,S)=\{(a_1,y):b_1\preceq y\preceq b_{2k+1-n}\}$ so
$|\FEBP(1,S)|=2k+1-n$.
\item $\LEBP(1,S)=\{(a_{k+1},b_{k+1})\}$ so $|\LEBP(1,S)|=1$.
\item When $n<2k$, $\DFEL(1,S)$ is a maximal non-reversible set
which is smaller than $S$.  Also, $\DLEF(1,S)$
is a non-maximal reversible set.
\item When $n=2k$, both $\DFEL(1,S)$ and $\DLEF(1,S)$ are non-maximal
reversible sets.
\end{enumerate}
\end{example}

\begin{example}\label{exa:k ge n-extreme}
Let $(n,k)$ be a pair with $n\leq k$. Set
\[S:=\{(x,y): a \preceq x \preceq y \preceq b \text{ for some } (a,b)\in \Max(S)\},\] where
\[
\Max(S)= \{(a_2,b_{k+1}),(a_1,b_{k+2-n}), (a_{k+2},b_{k+2})\}.
\]
The set $S$ is non-reversible as it contains the strict alternating
cycle
\[
C^*=\{(a_1,b_1),(a_2,b_{k+1}), (a_{k+2},b_{k+2})\}, 
\]
first introduced in Example~\ref{exa:k ge n}.
Then the following statements hold.

\begin{enumerate}
\item $S\in\MINR(n,k)$ and $|S|=(k+1)(k+2)/2+2-n$.  
\item $\FEBP(2,S)=\{(a_2,w):b_{k+3-n}\preceq w\preceq b_{k+1}\}$ so 
$|\FEBP(1,S)|=n-1$.
\item $\LEBP(2,S)=\{(a_{k+2},b_{k+2}\}$ so 
$|\LEBP(k+2,S)|=1$.
\item The set $\DLEF(2,S)$ is the canonical reversible set consisting of all 
$(x,y)\in\Inc(A,B)$ with $a_1 \preceq x \preceq y \preceq b_{k+1}$. In particular, $\DLEF(2,S)$ is the canonical reversible set associated with the $h$-contiguous sequence $(a_1, a_2, \ldots, a_{k+1})$.
\end{enumerate}
\end{example}

The notions of a contraction blocking pair and an expansion blocking
pair are dual in the natural symmetry on the crown, i.e., if $\phi$
is the automorphism of the crown $S_n^k$ discussed in Section~\ref{sec:auto}, then $((a,b),(x,y))$ is a contraction blocking pair at $i$ if and only if $((\phi(a),\phi(b)),(\phi(x),\phi(y)))$ is an expansion blocking pair at~$-i$.  Accordingly, the pitfalls identified for expansions in the preceding examples can also occur for contractions.

\section{Independent, Non-Reversible Sets: $k<n\leq 2k$}\label{sec:part-2}

This is the second of three sections devoted to the study of
independent, non-reversible sets.  However, in this section,
we only consider pairs $(n,k)$ with $k<n\le 2k$ and prove the
inequality in Theorem~\ref{thm:main-2}: when $k<n\le 2k$, if $S\in \INR(n,k)$, then $|S|\le 2+(2k-n+2)(2k-n+1)/2$.  In fact, we do much more.  We completely determine the family $\MINR(n,k)$ of all maximal independent, non-reversible sets.  Consequently, we know all their possible sizes. Theorem~\ref{thm:main-2} simply extracts the largest value among them. 

For the remainder of this section, fix a pair $(n,k)$ with $k<n\le 2k$ and use the abbreviations $\MINR=\MINR(n,k)$, $\NINR=\NINR(n,k)$ and $\INR = \INR(n,k)$.

From Lemma~\ref{lem:small-sac},
we know that every set $S\in\MINR$ contains a strict alternating cycle
of size~$3$.  In the range $k<n\le 2k$, we can say more.

\begin{proposition}\label{pro:SAC-k < n}
If $S\in\INR$, then every strict alternating cycle contained in
$S$ has size~$3$.
\end{proposition}

\begin{proof}  Fix a set $S\in \INR$ and a strict alternating cycle $C$ of size $m$ in $S$. Since $S$ is independent, $m\neq 2$. Toward a contradiction, suppose $m\ge 4$.  By applying Lemma~\ref{lem:dots} to the pair $(S,C)$, we conclude that $4n\le mn\le 2(n+k)$. This implies $n\le k$, a contradiction.
\end{proof}

From Proposition~\ref{lem:small-sac}, we know that $\MINR=\MINR_{D3}\cup\MINR_{O3}$. Furthermore, Proposition~\ref{pro:D3-O3} implies $\MINR_{O3}=
\{\phi(S):S\in\MINR_{D3}\}$, so  both $\MINR_{D3}$ and $\MINR_{O3}$ are non-empty when $n\leq 2k$.  
However, it is not clear from their definitions that  
$\MINR_{D3}$ and $\MINR_{O3}$ are disjoint.  This is a detail we will discover.

Recall that for any $(a,b),(x,y)\in \Inc(A,B),$ we say $(x,y)$ is contained in $(a,b)$ provided $a\preceq x\preceq y \preceq b$. The result is an inclusion order on $\Inc(A,B)$. In the descriptions to follow, we will reference down sets and up sets in this poset on $\Inc(A,B)$.

We show that every set $S\in\MINR_{D3}$ is a down set in $\Inc(A,B)$.   Using duality, every set in $\MINR_{O3}$ is
an up set in $\Inc(A,B)$.  In contrast, there are canonical
reversible sets which are neither up
sets nor down sets in $\Inc(A,B)$.  To see this, consider the crown
$S_4^5$ and the canonical reversible set in Example~\ref{exa:max-reverse}.
The set contains $(a_7,b_9)$ but does
not contain either $(a_7,b_8)$ or $(a_6,b_9)$.  In the next section, when $n\le k$, the
analysis of maximal, non-reversible sets is considerably more complex as they too may be neither
down sets nor up sets in $\Inc(A,B)$.

If we fully understand either of the two subfamilies $\MINR_{D3}$ and
$\MINR_{O3}$, then we have all information for $\MINR$. So, for the remainder of this
subsection, we focus on $\MINR_{D3}$ and emphasize our restriction to the case $k<n\le 2k$.

At several steps in the discussion to follow, we will need the next proposition.

\begin{proposition}\label{pro:2-around}
There do not exist pairs $(v_1,v_2)$ and $(v_3,v_4)$ of points on the circle, each of size at most $k+1$, so that (1)~$v_1 \preceq v_3 \preceq v_2$ and $v_1 \preceq v_4 \preceq v_2$ and (2)~for each $i\in[n+k]$,
$u_i$ is in $(v_1,v_2)$ or $u_i$ is in $(v_3,v_4)$.
\end{proposition}

\begin{proof}
Suppose to the contrary that two such pairs exist. By (1), 
$(v_1,v_2)$ and $(v_3,v_4)$ share at least 2 points. Since each has size at most $k+1$, (2) yields the inequality $n+k\leq 2(k+1) -2 = 2k$. However, this implies $n\leq k$, a contradiction.
\end{proof}

Next, we give a construction for sets belonging to $\MINR_{D3}$.
Fix a positive integer $t$ and let
$C_0=\{(x_\alpha,y_\alpha): \alpha\in [2t+1]\}$
be an alternating cycle (in general, not strict) in $\Inc(A,B)$ such
that for each $\alpha\in[2t+1]$, the following two conditions
are satisfied:

\begin{enumerate}
\item $x_\alpha\preceq y_\alpha \prec x_{\alpha+1}\preceq y_{\alpha+1}$; and
\item the size of $(x_\alpha,y_{\alpha+t})$ is $k+1$.
\end{enumerate}

In the discussion to follow, we refer to these two conditions
as the \textit{Matching Conditions}.
Given an alternating cycle $C_0$ satisfying the Matching Conditions, we then let
$D(C_0)$ be the down set of $C_0$, i.e. all pairs $(x,y)\in\Inc(A,B)$
for which there is an $\alpha\in [2t+1]$ with
$x_\alpha\preceq x \preceq y \preceq y_\alpha$. 
If we let
$s_\alpha$ denote the size of $(x_\alpha,y_\alpha)$, then
 $|D(C_0)|=\sum_{\alpha=1}^{2t+1}s_\alpha(s_\alpha+1)/2$.

\begin{example}\label{exa:AC}
When $n=47$ and $k=42$, the following incomparable pairs form an alternating cycle
$C_0$ of size~$7$ satisfying the Matching Conditions for $t=3$:
\[
C_0=\{(a_1,b_4),(a_{13},b_{20}),(a_{25},b_{31}),(a_{37},b_{43}),
(a_{51},b_{55}),
 (a_{67},b_{67}),(a_{78},b_{79})\}.
\]

It follows that
\[
|D(C_0)|=\binom{5}{2}+\binom{9}{2}+\binom{8}{2}+\binom{8}{2}+\binom{6}{2}
+\binom{2}{2}+\binom{3}{2}.
\]
\end{example}

These next two exercises are left for the reader. For the first, one may consider the $2t+1$ pairs $(a_{\alpha+t+1}, b_\alpha)$ of length $n$ with the property that, for any point $u$ on the circle, there are at most $t+1$ values $\alpha$ for which $a_{\alpha+t+1}\preceq u \preceq b_\alpha$. 

\begin{exercise}
If $k<n\le 2k$ and $t$ is a positive integer, then there
is an alternating cycle $C_0$ in $\Inc(A,B)$ of size~$2t+1$ which satisfies
the Matching Conditions if and only if $t(n-k)\le k$.
\end{exercise}

\begin{exercise}\label{ex:sum}
If $k<n\le 2k$, $t(n-k)\leq k$, and $(s_1,s_2,\dots,s_{2t+1})$ is
a sequence of positive integers, then there is
an alternating cycle $C_0=\{(x_\alpha,y_\alpha):\alpha\in[2t+1]\}$
satisfying the Matching Conditions with the size of $(x_\alpha,y_\alpha)$
equal to $s_\alpha$ for each $\alpha\in[2t+1]$ if and only if
$\sum_{\alpha=1}^{2t+1}s_\alpha =k+2t+1-t(n-k)$.
\end{exercise}

As an illustration of the preceding exercise, we note that
in Example~\ref{exa:AC},  $n=47$ and $k=42$ so that the maximum
value of $t$ is $8$. We chose $t=3$ in which case the
sum of the sizes of the seven pairs in $C_0$ was $42+6+1-3\cdot5=34$.  
Note that when $n=2k$, the maximum value of
$t$ is~$1$ and all pairs must have size 1 as illustrated in Example~\ref{exa:k < n}.

\begin{lemma}\label{lem:membership}
Let $t\ge1$ and let $C_0=\{(x_\alpha,y_{\alpha}):\alpha\in [2t+1]\}$
be an alternating cycle satisfying the Matching Conditions. Then the down set $D(C_0)$ is in $\MINR_{D3}$.
\end{lemma}

\begin{proof}
Any two distinct pairs $(x,y)$ and
$(x',y')$ in $D(C_0)$ are non-adjacent since there is some $\alpha\in [2t+1]$ so that both $(x,y)$ and $(x',y')$ are contained
in $(x_\alpha,y_{\alpha+t})$ which has size $k+1$. Also, $S$ is non-reversible
as it contains the alternating cycle $C_0$.
So $D(C_0)\in\INR$.
Since there is a subset of $C_0$ which
constitutes a strict alternating cycle of size~$3$ by Proposition~\ref{pro:SAC-k < n}, it is clear that
this strict alternating cycle satisfies the Disjoint Property.

To complete the proof, we need only show that $D(C_0)$ is a maximal independent set.
Let $(a,b)$ be a pair in $\Inc(A,B)$ which does not belong to
$D(C_0)$.  We show there is some
$(x,y)$ in $D(C_0)$ which is adjacent to $(a,b)$ in $G_n^k$.

Suppose first that there is some $\alpha\in [2t+1]$ with $x_\alpha\preceq a \preceq y_\alpha$. Since $(a,b)\not\in D(C_0)$, we have
$a\preceq y_\alpha\prec b$.  If $a \preceq y_\alpha \prec x_{\alpha+t+1}\preceq b $, then $(a,b)$ and $(x_{\alpha+t+1},y_\alpha)$
violate Proposition~\ref{pro:2-around}. If $a \preceq y_\alpha\prec b \prec x_{\alpha+t+1}$, then $(a,b)$ is
adjacent to $(x_{\alpha+t+1},y_{\alpha+t+1})$.

So we may assume that there is some $\alpha$ where $y_\alpha\prec a\prec x_{\alpha+1}$.
If $a\preceq b\prec x_{\alpha+t+1}$, then $(a,b)$ is adjacent to $(x_{\alpha+t+1},
y_{\alpha+t+1})$.  If $x_{\alpha+t+1}\preceq b\prec x_\alpha$, then $(a,b)$ is adjacent
to $(x_\alpha,y_\alpha)$.  These observations complete the proof of the
lemma.
\end{proof}

The next result asserts that the construction we have just
presented actually defines the family $\MINR_{D3}$.

\begin{theorem}\label{thm:D(C_0)}
If $k<n\le 2k$ and $S\in\MINR_{D3}$, then there is
some $t\ge1$ and an alternating cycle $C_0$ of size $2t+1$ which satisfies the Matching Conditions and has the property that $S=D(C_0)$.
\end{theorem}

\begin{proof}
Fix $S\in\MINR_{D3}$.  The \textit{span} of a strict alternating cycle is the sum of the sizes of the incomparable pairs which define the cycle. Of all the strict alternating cycles of
size~$3$ contained in $S$ which satisfy the Disjoint Property,
we choose one with maximum span. Call this cycle $C:=\{(x_1,y_1), (x_2, y_2), (x_3,y_3)\}$. We know establish 3 claims needed to complete the proof of Theorem~\ref{thm:D(C_0)}.

\begin{claim}\label{pro:k<n}
For each $\alpha\in \{1,2,3\}$, the following two statements hold:

\begin{enumerate}
\item $x<y_\alpha$ in $S_n^k$ for all $x$ with $y_\alpha \prec x\preceq y_{\alpha+1}$; and
\item $x_\alpha<y$ in $S_n^k$ for all $y$ with $x_{\alpha+2}\preceq y \prec x_\alpha$.
\end{enumerate}
\end{claim}

\begin{proof}
Fix $x$ with $y_\alpha \prec x\preceq y_{\alpha+1}$ and suppose, toward a contradiction, that $(x,y_\alpha)\in\Inc(A,B)$. Since $(x_\alpha, y_{\alpha+1})\in \Inc(A,B)$, both $(x,y_\alpha)$ and $(x_\alpha, y_{\alpha+1})$ have size at most $k+1$. Because $C$ satisfies the Disjoint Property, we have $x_\alpha \preceq y_\alpha \prec y_{\alpha+1}$. As a result, $(x,y_\alpha)$ and $(x_\alpha,y_{\alpha+1})$ violate Proposition~\ref{pro:2-around}. 

For the second statement, suppose $(x_\alpha, y)\in \Inc(A,B)$. A similar contradiction is reached by considering the pairs $(x_\alpha,y)$ and 
$(x_{\alpha+2},y_\alpha)$.
\end{proof}

Considering $S$ as a subposet of $\Inc(A,B)$ ordered by inclusion,
let $\Max(S)$ denote the maximal elements of $S$.

\begin{claim}\label{lem:max-elements}
The pairs in $C$ belong to $\Max(S)$. Furthermore, if
$(x,y)$ and $(x',y')$ are distinct pairs in $\Max(S)$,
then they are disjoint.
\end{claim}

\begin{proof}
We first show that any pair $(x,y)\in\Max(S)$ overlaps
at most one of the pairs in $C$.
If $(x,y)\in \Max(S)$ and $(x,y)$ overlaps all three
pairs in $C$, then there is
some $\alpha$ for which $(y_\alpha,x_{\alpha+2})$ is contained
in $(x,y)$.  However, $(x_{\alpha+2},y_\alpha)$ and $(x,y)$ then
violate Proposition~\ref{pro:2-around}.

Now suppose that there is a pair $(x,y)\in \Max(S)$ that overlaps $(x_\alpha,y_\alpha)$ and
$(x_{\alpha+1},y_{\alpha+1})$ but not $(x_{\alpha+2},y_{\alpha+2})$.
Using the previous claim, this implies
that $(x,y)$ is adjacent to $(x_{\alpha+2},y_{\alpha+2})$ in
$G_n^k$.  The contradiction completes the proof that a
pair $(x,y)\in\Max(S)$ overlaps at most one pair in $C$.

Now suppose that some $(x_\alpha,y_\alpha)$ is not in $\Max(S)$.
Then there is a pair $(x,y)\in\Max(S)$ so that
$(x_\alpha,y_\alpha)$ is properly contained in $(x,y)$. Because $(x,y)$ does not overlap any other pairs in $C$, it follows that we
can replace $(x_\alpha,y_\alpha)$ in $C$ with $(x,y)$ and obtain
a strict alternating cycle $C'$ whose span is larger than the
span of $C$.  The contradiction shows that the pairs of
$C$ belong to $\Max(S)$.

We now show that all pairs in $\Max(S)$ are disjoint. To do this, we first show that no pair in $\Max(S)-C$ overlaps a pair in $C$. Toward a contradiction, suppose $(x,y)\in \Max(S)-C$ overlaps $(x_1, y_1)\in C$. Thus far we know that either $x_1 \prec x \preceq y_1 \prec y$ or $x \prec x_1 \preceq y \prec y_1$. We supply the argument only for the first, as  a similar proof holds for the second. Because $(x,y)$ overlaps at most one pair in $C$, we have  $x_1 \prec x \preceq y_1 \prec y \prec x_2 \preceq y_2$. Since the size of $(x_1,y)$ is less than the size of incomparable pair $(x_1,y_2)$, we have $(x_1,y)\in \Inc(A,B)$. However, $(x,y)$ and $(x_1,y_1)$ are both in $\Max(S)$ so $(x_1,y) \not\in S$. Since $S$ is a maximal independent set,
it follows that there is some $(a,b)\in S$ with $(a,b)$
adjacent to $(x_1,y)$ in $G_n^k$.  Then we must have
$(a,y_1)\in\Inc(A,B)$; otherwise $(a,b)$ is
adjacent to $(x_1,y_1)$ in $G_n^k$. Similarly, we must have
$(x,b)\in\Inc(A,B)$; otherwise $(a,b)$ is adjacent to $(x,y)$.  As a result, the intervals $(a,y_1)$ and $(x,b)$ violate Proposition~\ref{pro:2-around}, a contradiction. Therefore, the intervals in $\Max(S)-C$ are disjoint from the intervals in $C$. 

Finally, we consider two distinct pairs $(x,y)$ and $(x',y')$ 
in $\Max(S)-C$. Toward a contradiction, suppose that they overlap with $x \prec x' \preceq y \prec y'$. We have already established that neither overlaps a pair in $C$. So we may assume $x_1 \preceq y_1 \prec x \prec x' \preceq y \prec y' \prec x_2 \preceq y_2$. Because $(x_1, y_2)\in \Inc(A,B)$, we may conclude $(x,y')\in \Inc(A,B)$. Furthermore, $(x,y')\not\in S$ because $(x,y)\in \Max(S)$. As before, this implies there exists $(a,b)\in S$ with $(x,y')$ adjacent to $(a,b)$ in $G_n^k$. The same argument yields a violation of Proposition~\ref{pro:2-around}, a contradiction. Thus all pairs in $\Max(S)$ are disjoint, completing the proof of the claim. 
\end{proof}

\begin{claim}\label{lem:downset}
The pairs in $S$ form a downset in $\Inc(A,B)$.
\end{claim}

\begin{proof}
Let $(x,y)\in\Max(S)$ and let $(x',y')$ be any pair
from $\Inc(A,B)$ contained in $(x,y)$. 
Toward a contradiction, suppose $(x',y')$ is not in $S$. Then there is some $(a,b)\in S$ with $(a,b)$ adjacent to
$(x',y')$ in $G_n^k$.

Choose a maximal element $(z,w)\in\Max(S)$ with
$(a,b)\subseteq(z,w)$. Since $(x,y)\in \Inc(A,B)$ while $(a,b)$ and $(x',y')$ are adjacent, we do not have $(z,w)=(x,y)$. So $(z,w)$ and $(x,y)$ are disjoint by the previous claim. This, together with the fact that $x'<b$ and $a<y'$, implies that $(x,y)$ and $(z,w)$ are adjacent, a contradiction to the fact that $S$ is an independent set.
\end{proof}

Let $m=|\Max(S)|$ and label the elements of $\Max(S)$
as $\{(z_\beta,w_\beta):1\le\beta\le m\}$ so that
\[
z_1\preceq w_1\prec z_2\preceq w_2\prec z_3\preceq w_3 \prec \dots\prec z_m\preceq w_m.
\]

Let $\alpha\in[m]$ and let $(z_\alpha,w_\alpha)=(a_i,b_j)$.
Then we must have $(a_i,b_j)\in\FEBP(i,S)\cap\LEBP(j-k,S)$.
It follows immediately that there are
distinct integers $\beta,\gamma\in[m]$ so that
$((z_\alpha,w_{\alpha}),(z_\beta,w_\beta))$ is an
expansion blocking pair at $i$ and $((z_\gamma,w_\gamma),
(z_\alpha,w_\alpha))$ is an expansion blocking pair at $j-k$. Furthermore, we may conclude that $\gamma=\beta+1$ since otherwise $(z_{\beta+1},w_{\beta+1})$ is adjacent to $(z_\alpha, w_\alpha)$. Each pair in $\Max(S)$ participates in two blocking pairs like this, so we may deduce that there is some $t$ for which $\beta=\alpha+t$ and $\gamma+t=\alpha$ wherein $m=2t+1$ and the size of $(z_\alpha,w_{\alpha+t})$ is $k+1$. This implies that $C_0:=\Max(S)$ is indeed a cycle which satisfies the Matching Conditions and $S=D(C_0)$.
These observations complete the proof of Theorem~\ref{thm:D(C_0)}.
\end{proof}

Having established the form of all sets in $\MINR_{D3}$,
we proceed with the task of analyzing the possible sizes. Let $t$ be a positive integer and let $C_0=\{(x_\alpha,y_\alpha):\alpha\in [2t+1]\}$ be an
alternating cycle of size $2t+1$ satisfying the
Matching Conditions.  As we have already noted,

\[
|D(C_0)|=\sum_{\alpha=1}^{2t+1}\frac{(s_\alpha+1)s_\alpha}{2}.
\]
Because Exercise~\ref{ex:sum} shows $\sum_{\alpha=1}^{2t+1}s_\alpha =k+2t+1-t(n-k)$, a simple calculation shows that $|D(C_0)|$ is maximized
when there is at most one value of $s_\alpha$ that is
larger than~$1$.  Accordingly, when $t$ is fixed, the
maximum value of $|D(C_0)|$ is attained when
(1)~$x_\alpha=a_{\alpha k+1}$ and $y_\alpha=b_{\alpha k+1}$ for all
$\alpha\in[2t]$ and (2)~$x_{2t+1}=a_1$ and
$y_{2t+1}=b_{k+1-t(n-k)}$.  (Note that $1\leq k+1-t(n-k)\leq k$.) With these values, $s_\alpha=1$ for
all $\alpha\in[2t]$, while $s_{2t+1}=k+1-t(n-k)$.
This yields $|D(C_0)|= 2t+(k+1-t(n-k))(k+2-t(n-k))/2$ which is maximized when $t=1$, resulting in the value
$2+(2k+1-n)(2k+2-n)/2$.  With this observation, the proof
of Theorem~\ref{thm:main-2} is complete.

We comment that when $k<n\le 2k$, there is
essentially only one extremal example.  To be more precise,
for
\[
C_0=\{(a_1,b_{2k-n+1}),(a_{k+1},b_{k+1}), (a_{2k+1},b_{2k+1})\},
\]
 each set in $\NINR$ is obtained from $D(C_0)$ via
the natural symmetries $\phi$ and $\tau$.

\section{Independent, Non-Reversible Sets: $n\leq k$}\label{sec:part-3}

This is the third and last of the sections devoted to the study of
independent, non-reversible sets.  Here, we will prove
Theorem~\ref{thm:main-3}: if $n\le k$ and $S$ is an independent, non-reversible set in 
$G_n^k$, then $|S|\le (k+1)(k+2)/2+2-n$.  With the construction
given in Example~\ref{exa:k ge n-extreme}, this inequality is best possible.  As noted previously,
we are actually able to determine the entire family
$\NINR(n,k)$ of all maximum size independent, non-reversible sets,
but due to space limitations, we restrict our discussion to
determining the common size of the posets in this family.

We consider Theorem~\ref{thm:main-3} to be our capstone result, with the 
difficulty rooted in the fact that the family $\MINR(n,k)$ of maximal 
independent, non-reversible sets is much more complicated when $n\le k$.  
To illustrate this complexity,  
the following anomaly shows the existence of sets in $\MINR(n,k)$ which are neither up sets nor down sets in $\Inc(A,B)$ as opposed to the structure found when $k<n\leq 2k$ in the previous section. (The proof of this result is left as an exercise.)

\begin{proposition}\label{pro:complexity}
Let $m$ be a positive integer and let $J\subseteq [m]$ be arbitrary.
Then there are pairs $(n,k)$ with $n\le k$, a set $S\in\MINR(n,k)$, and
a set $\{(x_i,y_i):1\le i\le m\}\subseteq\Inc(A,B)$ such
that (1)~$(x_i,y_i)$ is contained in $(x_{i+1},y_{i+1})$ for all
$i\in[m-1]$ and (2)~$(x_i,y_i)\in S$ if and only if $i\in J$.
\end{proposition}

\subsection{Further Details on Canonical Reversible Sets}

Ironically, in order to prove a result
about independent, non-reversible sets, we must first go back and
study canonical reversible sets in greater detail.

The results of this subsection apply to any pair $(n,k)$.
Let $T$ be a canonical reversible set.  Then there
is a uniquely determined $h$-contiguous sequence
$\sigma=(x_1,x_2,\dots,x_{k+1})$ for which $T=T(\sigma)$.
We refer to $x_1$ as the \textit{base} element of $\sigma$.
When $i\ge2$, an element $x_i$ is called a \textit{leading} element of
$\sigma$ when $x_i$ is the first element in the contiguous set $\{x_1,x_2,\dots,x_i\}$. Dually, $x_i$ is referred to as a \textit{trailing} element of
$\sigma$ when $x_i$ is the last element of the
contiguous set $\{x_1,x_2,\dots,x_i\}$.

It is worth noting that that there are $(n+k)2^k$ canonical
reversible sets in $G_n^k$ since there are $n+k$ choices
for the base element of $\sigma$.  Then, for each $i\ge2$, there
are two choices: either $x_i$ is leading or trailing.

Recall that for each $a_i\in A$, we have $I(a_i)=\{b_i,b_{i+1},\dots,
b_{i+k}\}$.  Respecting their appearance on the circle, $b_i$ is
the first element of $I(a_i)$ and $b_{i+k}$ is the last.  Dually,
$a_{i-k}$ is the first element of $I(b_i)$ and $a_i$ is the last.  In the same sense, when $0\le r\le k$, a subset of the
form $\{b_i,b_{i+1},\dots,b_{i+r}\}$ is an \textit{initial portion}
of $I(a_i)$ while $\{b_{i+r}, b_{i+r+1},\dots,b_{i+k}\}$ is
a \textit{terminal portion}. Analogously, an initial portion of $I(b_i)=\{a_{i-k}, \ldots, a_i\}$ has the form $\{a_{i-k}, \ldots, a_{i-r}\}$ while a terminal portion is of the form $\{a_{i-r}, \ldots, a_{i}\}$, where $0\leq r \leq k$.

The following proposition is an easy exercise, following essentially
from the definition of a canonical reversible set. 

\begin{proposition}\label{pro:init-term}
Let $\sigma=(x_1,x_2,\dots,x_{k+1})$ be an $h$-contiguous sequence
and let $T=T(\sigma)$ be the canonical reversible set
associated with $\sigma$. Then the following statements hold.

\begin{enumerate}
\item $B(x_1,T)=I(x_1)$.
\item If $i\ge2$ and $x_i$ is a leading element of $\sigma$,
then $B(x_i,T)$ is the initial portion of $I(x_i)$ with $k+2-i$ elements.
\item If $i\ge2$ and $x_i$ is a trailing element of $\sigma$,
then $B(x_i,T)$ is the terminal portion of $I(x_i)$ with $k+2-i$ elements.
\end{enumerate}
\end{proposition}

By symmetry, it follows that for each $y\in B(T)$,  the set $A(y,T)$ is either
an initial or terminal portion of $I(y)$.  

Below are two elementary exercises. The first is a consequence of the order of $x$ and $x'$ in the associated $h$-contiguous sequence.

\begin{proposition}\label{pro:contain}
Let $T$ be a canonical reversible set, and let $(x,y),(x',y')\in T$.
If $x\preceq y\prec x'\preceq y'$ and $x'<y$ in $S_n^k$, then $(x,y')\in T$.
\end{proposition}

\begin{proposition}\label{pro:BPs}
Let $T$ be a canonical reversible set, and let $(a_i,b_j)$ be
a pair in $T$. Then the following statements hold.
\begin{enumerate}
\item If $(a_i,b_j)\in\FEBP(i,T)$, then $(a_{j+1},b_{i+k})\in\LEBP(i,T)$.%check
\item If $(a_i,b_j)\in\LEBP(j,T)$, then $(a_{j-k},b_{i-1})\in\FEBP(j,T)$. 
\item If $(a_i,b_j)\in\FCBP(i,T)$, then $(a_{j+n-1},b_{i})\in\LCBP(i,T)$. %check
\item If $(a_i,b_j)\in\LCBP(j,T)$, then $(a_j,b_{i-n+1})\in\FCBP(j,T)$. %check
\end{enumerate}
\end{proposition}

The next result will be quite useful in upcoming proofs.

\begin{lemma}\label{lem:T}
Let $C=\{(x_\alpha,y_\alpha):\alpha\in[3]\}$ be a strict alternating
cycle in $\Inc(A,B)$ that satisfies the Disjoint Property.
If $T$ is a canonical reversible set containing $(x_2,y_2)$ and $(x_3,y_3)$,
then there are integers $m,p\in[n+k]$ so that the following statements hold.

\begin{enumerate}
\item $A(y_2,T)$ is a terminal portion of $I(y_2)$
starting at $a_m$.
\item $B(x_3,T)$ is an initial portion of $I(x_3)$ ending 
at $b_p$.
\item $a_{p-k}\preceq a_m\preceq x_2\preceq y_2\prec x_3\preceq y_3\preceq b_p\preceq b_{m+k}$.
\item  $(z,w)\in T$ whenever $a_m\preceq z\preceq y_2\preceq w\preceq b_{m+k}$.
\item  $(z,w)\in T$ whenever $a_{p-k}\preceq z\preceq x_3\preceq w\preceq b_p$.
\end{enumerate}
\end{lemma}

\begin{proof}
First we claim that $A(y_2,T)$ is
a terminal portion of $I(y_2)$. If instead $A(y_2,T)$ were an initial portion of $I(y_2)$, then $(x_1,y_2)\in T$ because $x_1$ is incomparable to $y_2$ and $x_1\prec x_2\preceq y_2$.   But $(x_1,y_2)$ is adjacent to $(x_3,y_3)$, a contradiction.  A symmetric argument shows that
$B(x_3,T)$ is an initial portion of $I(x_3)$.

Choose $m,p\in[n+k]$ so that $A(y_2,T)$ starts at $a_m$ and
$B(x_3,T)$ ends at $b_p$.  The first two statements of the lemma simply
reflect these choices.
Furthermore, since $(a_m,y_2)\in T$ and $(x_3,b_p)\in T$, Proposition~\ref{pro:contain} implies
$(a_m,b_p)\in T$. The fact that $a_m$ is incomparable to $b_p$ yields the ordering in the third statement of the lemma.

By the choice of $m$, it follows that $(a_m,y_2)\in\FEBP(m,T)$.  If $y_2=b_r$, then
Proposition~\ref{pro:BPs} implies $(a_{r+1},b_{m+k})\in\LEBP(m,T)$.  Clearly, $B(a_{r+1},T)$ is
an initial portion of $I(a_{r+1})$ ending at $b_{m+k}$ because of the blocking pair. Analogously, if  $x_3=a_s$, then $A(b_{s-1},T)$ is a terminal portion of $I(b_{s-1})$ starting at $a_{p-k}$.

Now, let $(z,w)$ be a pair with $a_m\preceq z\preceq y_2\preceq w\preceq b_{m+k}$.  We
claim that $(z,w)\in T$. This claim holds trivially if $w=y_2$. When $y_2 \prec  w\preceq b_{m+k}$, the claim is true by Proposition~\ref{pro:contain} because $T$ contains both $(z,b_r)$ and $(a_{r+1},w)$.  
The fifth statement of the lemma follows from a similar argument. 
%In particular, since $(x_3,b_p)\in\LEBP(p-k,T)$, we determine that $(a_{p-k},b_{s-1})\in\FEBP(p-k,T)$ where $x_3=a_s$. Thus $A(b_{s-1},T)$ is a terminal
%portion of $I(b_{s-1})$ starting at $a_{p-k}$.
\end{proof}

Before turning to the main body of the proof for Theorem~\ref{thm:main-3}, we require one more result that holds for any independent set. In particular, if
$S\in\mathbb{I}$, then we have the trivial inequality
$|B(a_1,S)|+|A(b_{k+2}, S)|\le k+1$.  But there are circumstances
in which we can sharpen this inequality.

\begin{lemma}\label{lem:k+3-n}
Let $S\in\mathbb{I}$.  If $S$ contains pairs $(a_1,y),(x,b_{k+2})$ with
$a_1\preceq y\prec x\preceq b_{k+2}$, then $|B(a_1,S)|+|A(b_{k+2},S)|\le k+3-n$.
\end{lemma}
 
\begin{proof}
Define a $3$-coloring of $[k+1]$ as follows.
Color $i$ red if $(a_1,b_i)\in S$; blue if $(a_{i+1},b_{k+2})\in S$; and
green if neither $(a_1,b_i)$ nor $(a_{i+1},b_{k+2})$ are in $S$.  The coloring
is well-defined as $(a_1,b_i)$ and $(a_{i+1},b_{k+2})$ are adjacent and cannot both belong to $S$.

With these definitions, it is clear that $|B(a_1,S)|+|A(b_{k+2},S)|$ is just
the number of integers in $[k+1]$ which have been
colored red or blue.  In the assumption of the lemma, if $y=b_i$ and $x=a_{j+1}$, then $i$ is colored red, $j$ is colored blue, and $i<j$.
Let $m$ be the largest integer for
which $m$ is red and there exists $j$, 
$m<j\le k+1$, such that $j$ is blue. Then let $p$ be the least
integer with $m<p$ such that $p$ is blue. Clearly,
all integers $r$ with $m<r<p$ are
green.  However, it is clear that $p\ge m+n-1$; otherwise $(a_1,b_m)$ is
adjacent to $(a_{p+1},b_{k+2})$.  So there are at least
$n-2$ elements of $[k+1]$ which are
green.  In turn, at most $(k+1)-(n-2)=k+3-n$ are red or
blue. 
\end{proof}

\subsection{The Main Body of the Proof of Theorem~\ref{thm:main-3}}

Fix a pair $(n,k)$ with $n\le k$. Since $\NINR_{D3}$ and $\NINR_{O3}$ are both non-empty (Lemma~\ref{lem:INR-exists} and Proposition~\ref{pro:D3-O3}), we choose to focus on a set $S\in\NINR_{D3}$. Toward a contradiction, assume
that $|S|>(k+1)(k+2)/2+2-n$. First we provide an outline of the proof.

Recall that $\NINR_{D3}$ is the family of all maximum size, independent, non-reversible sets in $G_n^k$ that contain a strict alternating cycle of size 3 which has the Disjoint Property. 
For each $S\in \NINR_{D3}$, let $\mathbb{C}_{D3}(S)$ denote the family of strict alternating cycles of size 3 in $S$ that have the Disjoint Property. Define $\mathbb{C}_{D3}$ to be the union of $\mathbb{C}_{D3}(S)$ taken over all $S\in \NINR_{D3}$.
Analogously, for each $C\in\mathbb{C}_{D3}$, we
let $\NINR_{D3}(C)$ denote the family of all $S\in\NINR_{D3}$ with $C\in \mathbb{C}_{D3}(S)$.

After establishing some tools, we show that
the following strict alternating cycle (first introduced in Examples~\ref{exa:k ge n}
and~\ref{exa:k ge n-extreme}) is in $\mathbb{C}_{D3}$:
\[
C^*=\{(a_1,b_1),(a_2,b_{k+1}),(a_{k+2},b_{k+2})\}.
\]
We then find a set in $\NINR_{D3}(C^*)$ such that, for each $(x,y)$ in the family, both $a_{1} \preceq x \preceq a_{k+2}$ and $b_1 \preceq y \preceq b_{k+2}$. Once this is accomplished, the following lemma gives the final contradiction.

\begin{lemma}\label{lem:size_contradiction}
Consider a set $S\in\INR$ which contains a strict alternating cycle $\{(x_\alpha,y_\alpha):\alpha\in
[3]\}$ in $\mathbb{C}_{D3}$ with  $x_1=a_1$ and $y_3=b_{k+2}$.
If $A(S)\subseteq \{a_1,a_2,\dots,a_{k+2}\}$ and
$B(S)\subseteq\{b_1,b_2,\dots,b_{k+2}\}$, then
$|S|\le (k+1)(k+2)/2+2-n$.
\end{lemma}

\begin{proof}
Partition $S$ as $M_1\cup M_2\cup M_3$, where 
$M_1$ consists of all $(a,b)\in S$ with $a=a_1$ or $b=b_{k+2}$,
$M_2$ consists of all $(a,b)\in S$ with $a_2\preceq a\preceq b\preceq b_{k+1}$,
and $M_3=S-(M_1\cup M_2)$. From Lemma~\ref{lem:k+3-n}, $|M_1|=|B(a_1,S)|+|A(b_{k+2},S)|
\le k+3-n$.  

To prove $|M_2|+|M_3|\le k(k+1)/2$, we will define
a 1--1 map $f$ which assigns to each $(a,b)\in M_3$ a pair $f(a,b)$ satisfying
the following two conditions: (1)~$f(a,b)$ is contained in $(a_2,b_{k+1})$
and (2)~$f(a,b)$ is adjacent to $(a,b)$ (so $f(a,b)\not \in M_2$).  
Since the number of pairs in $\Inc(A,B)$ which are contained in $(a_2,b_{k+1})$ is exactly $k(k+1)/2$, the inequality $|M_2|+|M_3|\le k(k+1)/2$ follows.

If $(a,b)\in M_3$, then there
are integers $i,j$ with $1\le i<j\le k+2$ so that $(a,b)=(a_j,b_i)$ by the hypothesis of the lemma. Note $i<j-2$ as $n\geq 3$.
We then define $f(a,b)=f(a_j,b_i)=(a_{i+1},b_{j-1})$.
It is easy to see that all requirements are met by this map.

The proof of the lemma is now complete since
% \begin{eqnarray*}
\[
|S|
=|M_1|+(|M_2|+|M_3|)
\le (k+3-n)+\frac{k(k+1)}{2}
=\frac{(k+1)(k+2)}{2}+2-n.
\]
%\end{eqnarray*}
%\begin{eqnarray*}
%|S|
%&=&|M_1|+(|M_2|+|M_3|)\\
%&\le& (k+3-n)+k(k+1)/2\\
%&=&(k+1)(k+2)/2+2-n.
%\end{eqnarray*}
\end{proof}

Now that we have an outline of the proof of Theorem~\ref{thm:main-3}, we begin with a lemma which gives flexibility in applying the contraction lemma (Lemma~\ref{lem:contraction-block}). This is the first step as we steer toward finding a set in $\NINR_{D3}$ which contains cycle $C^*$.

\begin{lemma}\label{lem:contract-free}
Let $S\in\NINR_{D3}$ and $i\in[n+k]$. Under the assumption that $|S|>(k+1)(k+2)/2+2-n$, both $\DFCL(i,S)$ and
$\DLCF(i,S)$ belong to $\NINR$.
\end{lemma}

\begin{proof}
First note that if $\FCBP(i,S)=\emptyset$, then $\DFCL(i,S) = \DLCF(i,S) = S$. If both $\DFCL(i,S)$ and $\DLCF(i,S)$ are non-reversible, then
the fact that $|\DFCL(i,S)|+|\DLCF(i,S)|=2|S|$ forces both to
belong to $\NINR$.  So if the lemma fails, one or both of
$\DFCL(i,S)$ and $\DLCF(i,S)$ must be reversible.  In this case, one of the following
two statements must apply:

\begin{enumerate}
\item $\DFCL(i,S)$ is reversible and $|\DFCL(i,S)|\ge |S|$, or
\item $\DLCF(i,S)$ is reversible and $|\DLCF(i,S)|\ge |S|$.
\end{enumerate}

We assume the first statement applies and argue to a contradiction.
From the details of the argument,  the proof for when the second statement holds is symmetric.

After a relabeling, we may assume $i=1$.  Let $R$ denote
the reversible set $\DFCL(1,S)$ and  
let $C=\{(x_\alpha,y_\alpha):\alpha\in[3]\}$ be any
member of $\mathbb{C}_{D3}(S)$.  Clearly, one of the pairs
in $C$ belongs to $\FEBP(1,S)$ because $R$ is reversible. Then relabel the pairs in $C$ so that $(x_1,y_1)\in\FEBP(1,S)$ and
$x_1=a_1$.

Let $T\in\MR$ with $R\subseteq T$ so that $|S|\le|R|\le |T|$. 
Since $|S|>(k+1)(k+2)/2+2-n$, it follows that $|T-R|<n-2$. We emphasize this inequality as many of our contradictions result from this fact.

If $T$ is not a canonical reversible
set, then Theorem~\ref{thm:main-1} implies
$|T|\le (k+1)(k+2)/2-n(n-1)/2+1$. As $|T| \geq |S|$, this contradicts our assumption about the size of $S$. 
It follows that $T$ is a canonical reversible set. 
As $T$ contains the pairs $(x_2,y_2)$ and $(x_3,y_3)$,  let $m,p\in[n+k]$ be the integers
specified by Lemma~\ref{lem:T}.
We now break our analysis into two cases based upon the location of $b_p$. 

\smallskip
\noindent \textit{Case 1:} Suppose $y_3\preceq b_p\prec b_{n+k}$. Say $y_1=b_j$.  Since $C$ is a strict alternating cycle, we have
 $(x_3,y_1)\in\Inc(A,B)$.  Therefore
$a_j\prec a_{j+n-1}\prec x_3$.
We further divide this case according to the location of $a_{p-k}$. 

First suppose $a_{p-k}\preceq a_{j+1}\preceq x_2$.  
By Lemma~\ref{lem:T}, it follows that $T$ contains all $n-1$ pairs in the following set:
\begin{equation*}
N=\{(a_{j+\gamma},b_p):\gamma\in[n-1]\}.
\end{equation*}
All pairs in $N$ are adjacent to $(x_1,y_1)$, so $N\subset T-S$.
Furthermore, since $b_p\neq b_{n+k}$, we know that
$N\cap R=\emptyset$ as all pairs in $R-S$ end at $b_{n+k}$.  However, this would imply that $|T-R|\ge n-1$,
which contradicts the inequality $|T-R| < n-2$.  

Therefore, it must be the case that
$a_{j+1}\prec a_{p-k}\preceq x_2$.
Let $t_3$ count the number of points $u_\delta$ on the circle with
$b_p\prec u_\delta\prec a_1$ and $t_1$ the number of points $u_\delta$ with $y_1\prec u_\delta\prec a_{p-k}$.  Because $b_p\prec b_{n+k}\prec a_1$, we have $t_3>0$. Similarly, $t_1>0$.  Now
let $s_1$ be the size of $(x_1,y_1)$.  Then $s_1>0$.
Since $s_1+t_1+t_3$ counts the number of points $u_\delta$ on the
circle with $b_p\prec u_\delta\prec a_{p-k}$, we know that
$s_1+t_1+t_3= n-1$.

Now consider the pairs in the following set which are contained in $T$:
\begin{equation*}
N'=\{(a_{m},b_{p+1-\beta}:\beta\in[s_1+t_1]\}\cup
\{(a_{p-k-1+\gamma},b_p):\gamma\in[s_1+t_3]\}.
\end{equation*}
Clearly, $|N'|=2s_1+t_1+t_3-1\ge s_1+t_1+t_3= n-1$.  Furthermore, all pairs
in $N'$ are adjacent to $(x_1,y_1)$ so $N'\subseteq T-S$.  Since
no pair in $N'$ ends at $b_{n+k}$, we conclude that $N'\cap R=\emptyset$.
As before, this implies $|T-R|\ge n-1$, a contradiction.

%The last two paragraphs together imply that the assumption that
%$y_3\preceq  b_p\prec b_{n+k}$ cannot hold.  Instead, it must be the case that 
\smallskip
\noindent \textit{Case 2:} Suppose $y_3\preceq  b_{n+k}\preceq  b_p$.  Since $C$ is a strict alternating cycle, we know
$(x_1,y_2)\in\Inc(A,B)$.  Then for every $\beta\in[n-1]$, we have
$y_2\prec b_{1-\beta}\preceq  b_{n+k}$. It follows
from Lemma~\ref{lem:T} that all $n-1$ pairs in the following set
belong to $T$: 
\begin{equation*}
N''=\{(x_2,b_{1-\beta}):\beta\in[n-1]\}.
\end{equation*}
All pairs in $N''$ are adjacent to $(x_1,y_1)$, so $N''\subseteq T-S$.
Furthermore, at most one pair from $N''$ is in $R$ as all pairs in $R-S$ end at $b_{n+k}$.  This implies that $|T-R|\ge n-2$, a contradiction.

 With contradictions in all cases, this completes the proof when the
first of the two statements applies. By reviewing the steps in the argument, the proof when the second of the two statements holds is symmetric. This completes the proof of the lemma.
\end{proof}

The natural next step is to include an analogous result for expansions. This would require a separate proof as we are working only with sets in $\NINR_{D3}$. However, such a result is not so straightforward. The following lemma for expansions shows that $\NINR_{D3}$ is closed under expansions with one possible exception. Although there is one new
wrinkle in the argument, many elements are quite similar to 
the preceding proof and, in those situations, we will be brief.

\begin{lemma}\label{lem:expand-ok}
Let $S\in\NINR_{D3}$ and $i\in[n+k]$. Under the assumption that $|S|>(k+1)(k+2)/2+2-n$, either $\DFEL(i,S)$ and $\DLEF(i,S)$ are both in $\NINR$ or one of the following holds: 
\begin{itemize}
\item $\DFEL(i,S)$ is reversible with $|\DFEL(i,S)|\ge |S|$ and each strict alternating cycle $\{(x_\alpha,y_\alpha): \alpha\in [3]\}$ in $\mathbb{C}_{D3}(S)$ has an $\alpha\in [3]$ such that $x_\alpha=a_i$ and $y_{\alpha+2} = b_{i+k+1}$. Furthermore, each of these cycles $C$ has a corresponding cycle $C'$ also in $S$ where $C'=\left(C - \{(x_{\alpha+1},y_{\alpha+1})\}\right) \cup \{(x_{\alpha+1},y)\}$ where the size of $(y,x_{\alpha+2})$ is 2.

\item $\DLEF(i,S)$ is reversible with $|\DLEF(i,S)|\ge |S|$ and each strict alternating cycle $\{(x_\alpha,y_\alpha): \alpha\in [3]\}$ in $\mathbb{C}_{D3}(S)$ has an $\alpha\in [3]$ such that $x_\alpha=a_{i-1}$ and $y_{\alpha+2} = b_{k+i}$. Furthermore, each of these cycles $C$ has a corresponding cycle $C'$ also in $S$ where $C'=\left(C - \{(x_{\alpha+1},y_{\alpha+1})\}\right) \cup \{(x,y_{\alpha+1})\}$ where the size of $(y_{\alpha},x)$ is 2.
\end{itemize}
\end{lemma}

\begin{proof}
If $\DFEL(i,S)$ or $\DLEF(i,S)$ does not belong to
$\NINR$, then one of the 
following two statements applies:

\begin{enumerate}
\item $\DFEL(i,S)$ is reversible and $|\DFEL(i,S)|\ge |S|$, or
\item $\DLEF(i,S)$ is reversible and $|\DLEF(i,S)|\ge |S|$.
\end{enumerate}

We show that if the first of these two statements holds, then we either obtain a contradiction or we discover the structure of the alternating cycles described in the first statement of the lemma. The proof for the second statement is analogous.

After a relabeling, we may assume $i=1$. If $\DFEL(1,S)$ is reversible, then every strict alternating cycle $\{(x_\alpha,y_\alpha):\alpha\in[3]\}$ in $\mathbb{C}_{D3}(S)$ must have $x_\alpha=a_1$ for some $\alpha\in[3]$. Arbitrarily fix a cycle $C\in \mathbb{C}_{D3}(S)$ and relabel its pairs so that $x_1=a_1$. 

Let $R$ denote the
reversible set $\DFEL(1,S)$ where $(x_1,y_1)\in
\FEBP(1,S)$. Choose $T\in\MR$ with $R\subseteq T$. 
Then  $T$ contains the pairs $(x_2,y_2)$ and $(x_3,y_3)$. Based on the size of $S$ and the relationship $|S| \le |R|\le |T|$, we conclude 
$T$ must be a canonical reversible
set wherein $|T-R| < n-2$.
Let $m,p\in[n+k]$ be the integers
specified by Lemma~\ref{lem:T}.

We pause to note that if $(x,y)\in R-S$, then $y=b_{k+2}$. As $S$ was maximal, then there must also 
be some $(a,b)\in S-R$ which is adjacent to $(x,y)$. Clearly, this
requires $(a,b)\in\FEBP(1,S)$, so $a=a_1$. 

We proceed by breaking the argument into cases based on the location of $b_{m+k}$.

\smallskip
\noindent \textit{Case 1:} Suppose $y_3\preceq b_{n+k}\preceq b_{m+k}$.
It follows that all $n-1$ pairs in the following set
belong to $T$: 
\begin{equation*}
M=\{(x_2,b_{1-\beta}):\beta\in[n-1]\}.
\end{equation*}
All pairs in $M$ are adjacent to $(x_1,y_1)$, so $M\subseteq T-S$.
However, at most one pair from $M$ belongs to $R-S$ since all pairs
in $R-S$ end at $b_{k+2}$.  This implies $|T-R|\ge n-2$, a contradiction.

\smallskip
\noindent \textit{Case 2:} Suppose $y_3\preceq b_{m+k}\prec b_{n+k}$. Let $y_1=b_j$. We further subdivide this case based on the location of $a_{p-k}$. 

First suppose $a_{j+1}\prec a_{p-k}\preceq a_m$.
Using the same definitions and notation from the proof of Lemma~\ref{lem:contract-free},
it follows that $T-S$ contains all pairs in the following set:
\begin{equation*}
N'=\{(a_{m},b_{p+1-\beta}:\beta\in[s_1+t_1]\}\cup
\{(a_{p-k-1+\gamma},b_p):\gamma\in[s_1+t_3]\}.
\end{equation*}
Because $a_1\preceq a_j\prec a_{j+1} \prec a_{p-k}$, we conclude $b_p\neq b_{k+2}$. So at most $1$ pair in $N'$ belongs
to $R$.  Since $|N'|\ge n-1$, this implies that $|T-R|\ge n-2$, a contradiction.  

Next consider the case where
$a_{p-k}\preceq a_{j+1}\preceq x_2$.
Observe that $T-S$ contains all $n-1$ pairs in the following set:
\begin{equation*}
M'=\{(a_{j+\gamma},b_p): \gamma\in [n-1]\}.
\end{equation*}
Because $x_1<y_3$ in $S_n^k$ and $x_1=a_1$, we have $b_{k+2} \preceq y_3 \preceq b_p \prec x_1$. If $b_p \neq b_{k+2}$, then $M' \not\subseteq R$ which implies $|T-R| \geq n-1$, a contradiction. 

Therefore, we may assume $y_3=b_{k+2}=b_p$. Recall the original choice of $C$ was arbitrary with pairs relabeled so that $x_1=a_1$. If $\DFEL(1,S)$ is not in $\NINR$, we explore the characteristics of cycles in $S$. 
Furthermore, if $x_3=a_\ell$, then $(x_2,b_{\ell-1})$ is in $T$ by Lemma~\ref{lem:T}. As $(x_2,b_{\ell-1})$ is not adjacent to any pair of the form $(a_1,y)$, we may conclude $(x_2,b_{\ell-1})\in S$. Thus $S$ contains the cycle $\{(a_1,y_1),(x_2,b_{\ell-1}),(a_\ell,b_{k+2})\}.$
A similar argument holds if $\DFEL(1,S)$ is not in $\NINR$. 
\end{proof}

%The reader may note that the preceding lemma depends heavily on the
%assumption that $|S|>(k+1)(k+2)/2+2-n$ as Example~\ref{exa:k ge n-extreme}
%shows that an expansion can transform that set into a canonical
%reversible set.  

Fix a cycle $C=\{(x_\alpha,y_\alpha):\alpha\in[3]\}$ in $\mathbb{C}_{D3}$.
For each $\alpha\in[3]$, 
the set $\{(x_\alpha,w): x_\alpha\preceq w\preceq y_\alpha\}$ is
called the \textit{forward $\alpha$-fan for $C$}.  Analogously, the
set $\{(z,y_\alpha)\in\Inc(A,B):x_\alpha\preceq z\preceq y_\alpha\}$ is  the \textit{backward $\alpha$-fan for $C$}.  In general,
when $S\in\NINR_{D3}(C)$ and $\alpha\in[3]$, there is no reason that $S$ should contain either the forward $\alpha$-fan or
the backward $\alpha$-fan for $C$. 
However, Lemma~\ref{lem:fan}, which follows from Lemma~\ref{lem:contract-free}, guarantees that $\NINR_{D3}$ contains sets with fans.

\begin{lemma}\label{lem:fan}
Let $C\in\mathbb{C}_{D3}$, and let $f:[3]\rightarrow\{\texttt{f},\texttt{b}\}$.
Then there is a set $S\in\NINR_{D3}(C)$ such that for each
$\alpha\in[3]$, $S$ contains the forward $\alpha$-fan for $C$ if
$f(\alpha)=\texttt{f}$ and $S$ contains the backward $\alpha$-fan for $C$
if $f(\alpha)=\texttt{b}$.
\end{lemma}

\begin{proof}
We simply start with any set $S \in \NINR_{D3}(C)$ and repeatedly apply the contraction
lemma, retaining the cycle $C$ at each step. By Lemma~\ref{lem:contract-free}, the result is another set in $\NINR_{D3}(C)$.  For example, say we desire a set with the forward
$1$-fan for $C$.  Clearly $(x_1,y_1)\in S$.
Of all points $u_i$ on the circle with (1)~$x_1\preceq u_i\preceq y_1$
and (2)~$(x_1,w)\in S$ for all $w$ with $u_i\preceq w\preceq y_1$, choose
$u_m$ as the unique one for which the size of $(x_1,u_m)$ is minimum.

If $x_1=a_m$, then $S$ contains the required forward $1$-fan and we
move on to other values of $\alpha$.  If instead $x_1\neq a_m$, then $(x_1,b_m)\in\LCBP(m,S)$.  Furthermore, $S'=\DFCL(m,S)$ contains
$C$, all pairs $(x_1,w)$ with $u_{m-1}\preceq w\preceq y_1$, and all
pairs belonging to the other two fans. We can repeatedly apply these modifications until
the desired set is obtained.
\end{proof}

With Lemmas~\ref{lem:contract-free}, ~\ref{lem:expand-ok}, and ~\ref{lem:fan} in hand, we proceed to show $C^*\in \mathbb{C}_{D3}$. 
For each $C\in\CZ$ with $C=\{(x_\alpha,y_\alpha):\alpha\in[3]\}$,
the pairs $\{(y_\alpha,x_{\alpha+1}):\alpha\in[3]\}$ are considered the
``gaps'' of $C$. Define $g_\alpha(C)$ to be the size of the
gap $(y_\alpha,x_{\alpha+1})$ and define the \textit{spread}
of $C$ to be the quantity
\begin{equation*}
\max\{g_\alpha(C)-g_{\alpha+1}(C)-g_{\alpha+2}(C): \alpha \in[3]\}.
\end{equation*}

In turn, we let $\mspread$ be the maximum value of the spread of $C$ taken
over all $C\in\mathbb{C}_{D3}$.  Since $2\le g_\alpha(C)\le n$
for all $C\in\mathbb{C}_{D3}$ and all $\alpha\in[3]$, then $\mspread$ is at most $n-4$.
Note that if $C^*\in\CZ$, then $\mspread=n-4$, but to establish this
will take some work. First we find a cycle $C\in \mathbb{C}_{D3}$ with spread $n-4$. 

\begin{claimn}
There is some $C\in\mathbb{C}_{D3}$ whose spread is $n-4$. i.e. $\mspread=n-4$.
\end{claimn}

%Fan lemma: fans to get g_3=n or (x_1,y_1) and (x_3,y_3) are both points.
%If trivial intervals, then both other gaps are not trivial.
%Gap (y_1,x_2) non-trivial. Expand unless 

\begin{proof}
Among all cycles in $\mathbb{C}_{D3}$, let $C$ be one with maximum spread. Fix $S\in \NINR_{D3}(C)$. Toward a contradiction, suppose the spread of $C$ is less than $n-4$. 
Label the pairs in $C$ so that the spread is $g_3(C)-g_1(C)-g_2(C)$.

%If $g_1(C)>2$, let $x_2=a_t$. Because $C$ has maximum spread, $(a_{t-1},y_2)\not\in S$. In particular, $(x_2,y_2)\in \FEBP(t,S)$. Observe $y_3\preceq b_{t+k} \preceq y_1$. 
%By Lemma~\ref{lem:expand-ok}, either $\DLEF(t,S)\in \NINR$ or each alternating cycle has $x_\alpha=a_t$ and $y_\alpha = b_{t+k}$. 

If $g_3(C)<n$, then use Lemma~\ref{lem:fan} to obtain a set $S'\in \NINR_{D3}(C)$ which contains the forward 3-fan for $C$ and the backward 1-fan for $C$. If $(x_1,y_1)=(a_i,b_j)$ with $i\neq j$, then $S'$ also contains the cycle $\{(a_{i+1},y_1), (x_2,y_2), (x_3,y_3)\}$ which has larger spread than $C$, a contradiction. So we may assume $(x_1,y_1)=(a_i,b_i)$ and, similarly, we may assume $(x_3,y_3)=(a_r,b_r)$. 

With $g_3(C)<n$, $(x_1,y_1)=(a_i,b_i)$, and $(x_3,y_3)=(a_r,b_r)$, it must be the case that $g_1(C)>2$. Otherwise $x_2<y_3$ in $S_n^k$, which is not possible because $C$ is a strict alternating cycle. Likewise, we may conclude $g_2(C)>2$. 

Because $C$ has maximum spread, we see $(a_{i},b_{i+1})\not\in S$. This implies $(x_1,y_1)\in \LEBP(i-k,S)$. As $g_3(C)<n\leq k$ and $(x_3,y_3)=(a_r,b_r)$, we conclude $x_3\neq a_{i-k}$. Furthermore, $x_2<y_1$, so $x_2\neq a_{i-k}$. Therefore, $\DFEL(i-k,S)$ contains the strict alternating cycle $\{(x_1,b_{i+1}),(x_2,y_2),(x_3,y_3)\}$ with spread larger than that of $C$. 
%As $C$ has no pair beginning at $a_{i-k}$, Lemma~\ref{lem:expand-ok} implies $\DFEL(i-k,S)\in \NINR$, a contradiction. 
Therefore, $\DFEL(i-k,S)\not\in \NINR$.
The only alternative, according to Lemma~\ref{lem:expand-ok}, is that each cycle in $\mathbb{C}_{D3}(S)$ has a gap of size $n$. However, $C$ has no such gap, a contradiction.

Lastly, suppose $g_3(C)=n$ and relabel the circle so $x_1=a_1$ and $y_3=b_{k+2}$. Because the spread of $C$ is less than $n-4$, either $g_1(C)>2$ or $g_2(C)>2$. First suppose $g_1(C)>2$. (An analogous argument holds if $g_2(C)>2$.) Because $C$ has maximum spread, $(x_2,y_2)=(a_\ell,y_2)\in S$ but $(a_{\ell-1},y_2)\not\in S$. Therefore, $(x_2,y_2)\in \FEBP(\ell,S)$. Now observe that $a_2\prec a_{\ell} \prec a_{k+2}$ and $x_2=a_\ell < y_1$ in $S_n^k$ imply $b_{k+2} \prec b_{k+\ell} \prec y_1$. Thus $y_3=b_{k+2} \neq b_{\ell+k}$ and $y_1\neq b_{\ell+k}$. Therefore, $\DLEF(\ell,S)$ contains the cycle $\{(x_1,y_1),(a_{\ell-1},y_2), (x_3,y_3)\}$ which has spread larger than that of $C$. By the choice of $C$, this implies $\DLEF(\ell,S)$ is not in $\NINR$.  As $\DLEF(\ell,S)\in \INR$, the first bullet point of Lemma~\ref{lem:expand-ok} must apply. So $g_1(C)=n$. Furthermore, there is another cycle $C'$ in $S$ with $g_3(C')=2$ and $g_2(C')=g_2(C)$. Thus the spread of $C$ is $n-g_1(C)-g_2(C)$ while the spread of $C'$ is $n-2-g_2(C)$, which is larger as $g_1(C)>2$. This contradicts the choice of $C$. 

With a contradiction in every case, we may conclude that $\mspread=n-4$. 
 %With $x_1=a_1$ and $y_3=b_{k+2}$, observe $a_1 \prec a_2 \prec a_{\ell}$ implies $y_3=b_{k+2} \prec b_{\ell+k+1} \prec y_1$. Since $C$ has no pair that ends at $b_{\ell+k+1}$, Lemma~\ref{lem:expand-ok} implies $\DLEF(\ell,S)$ is in $\NINR$. Further, $\DLEF(\ell,S)$ contains cycle $\{(x_1,y_1), (a_{\ell-1},y_2),(x_3,y_3)\}$ which has span larger than that of $C$, a contradiction. 
\end{proof}

Let $\mathbb{C}_{sp}$ denote the family of all $C\in\mathbb{C}_{D3}$ which have
spread $n-4$.  For each $C\in\mathbb{C}_{sp}$, we may assume that the 
pairs in $C$ have been labeled so that $a_1=x_1$ and $y_3=b_{k+2}$.  Let $\NINR_{sp}$ consist of all $S\in\NINR_{D3}$ for which there
is some $C\in\mathbb{C}_{sp}$ with $C$ contained in $S$.

\begin{claimn}
The strict alternating cycle $C^*=\{(a_1,b_1),(a_2,b_{k+1}),(a_{k+2},b_{k+2})\}$
is in $\mathbb{C}_{sp}$.
\end{claimn}

\begin{proof}
Choose $C\in\mathbb{C}_{sp}$ for which
the sum of the sizes of $(x_1,y_1)$ and $(x_3,y_3)$ is minimum.  
Note that $C=C^*$ if this sum is $2$, so we assume
the sum is larger than~$2$ and argue to a contradiction. We assume first
that the size of $(x_1,y_1)$ is at least~$2$.
Let $S\in\NINR_{sp}(C)$ contain the forward $1$-fan for $C$.

Let $y_1=b_j$.  Then $(x_1,b_{j-1})\in S$.  If $(a_j,y_2)\in S$,
we get an immediate contradiction, since $S$ contains the strict 
alternating cycle $C'$ obtained from $C$ by
replacing $(x_1,y_1)$ and $(x_2,y_2)$ by $(x_1,b_{j-1})$ and
$(a_j,y_2)$. 

It follows that $(x_2,y_2)=(a_{j+1},y_2)\in \FEBP(j+1,S)$. Because $j\geq 2$, we have $(x_3,y_3)\not \in \LEBP(j+1,S)$. Also, $y_1\neq b_{j+k+1}$ as $x_2<y_1$ in $S_n^k$, so $S'=\DLEF(j+1,S)$ contains $C'$. Therefore, $\DLEF(j+1,S)\not\in \NINR$, however $\DLEF(j+1,S)\in \INR$. By Lemma~\ref{lem:expand-ok}, the only alternative is for $\DFEL(j+1,S)$ to be reversible. In this case, since $x_2=a_{j+1}$, we must have $g_1(C)=n$, a contradiction.
%Since $S'\in\DLEF(j+1,S)$ contains $C'$, we have another contradiction.
%The argument is clearly symmetric when the size of $(x_3,y_3)$ is
%at least~$2$.  With this observation, the proof of
%Claim~2 is complete.
\end{proof}

Now that we have $C^*$ in a set $S$ from $\NINR_{D3}$, we can steer toward the hypotheses in Lemma~\ref{lem:size_contradiction}, seeking a set $S\in \NINR_{D3}$ with $A(S) \subseteq \{a_1,a_2,\dots,a_{k+2}\}$ and
$B(S)\subseteq\{b_1,b_2,\dots,b_{k+2}\}$. We establish one condition at a time.

\begin{claimn}\label{clm:A}
There is some $S\in\NINR(C^*)$ with $A(S)\subseteq\{a_1,a_2,\dots,a_{k+2}\}$.
\end{claimn}

\begin{proof}
Let $S$ be any set in $\NINR(C^*)$ which contains the backward
$2$-fan for $C^*$. For a proof by contradiction, suppose $(a,b)\in S$ with
$a_{k+2}\prec a\prec a_1$.  Since $S$ contains $(a_2,b_{k+1})$ and $S$ is independent, the intervals $(a,b)$
and $(a_2,b_{k+1})$ must overlap. Let $b=b_j$.  Then
$1\le j\le k$.  However, since $S$ contains the backward $2$-fan
for $C$, we have $(a_{j+1},b_{k+1})\in S$.  This is a contradiction
since $(a,b)$ and $(a_{j+1},b_{k+1})$ are adjacent.
\end{proof}

\begin{claimn}\label{clm:AB}
There is some $S\in\NINR(C^*)$ with $A(S)\subseteq\{a_1,a_2,\dots,a_{k+2}\}$
and $B(S)\subseteq\{b_1,b_2,\dots,b_{k+2}\}$.
\end{claimn}

\begin{proof}
We consider only sets $S\in\NINR(C^*)$ for which $A(S)\subseteq\{a_1,a_2,\dots,
a_{k+2}\}$, which exist by Claim~\ref{clm:A}.  Among these sets, let $S$ be one for which the number
of pairs $(a,b)$ with $b_{k+2}\prec b\prec b_1$ is minimum.  We suppose this minimum 
number is positive and argue to a contradiction.  

Of all pairs $(a,b)\in S$ with $b_{k+2}\prec b\prec b_1$, choose one of 
minimum size.  Then there is some $i$ with $3\le i\le k+1$ so that $a=a_i$ and $(a_i,b)\in\FCBP(i,S)$.

We note that any pair $(x,y)\in \LCBP(i,S)$ has $a_1\preceq x \preceq a_{k+2}$.  Furthermore, all pairs in $\DFCL(i,S)-S$
end at $b_{i-1}$.   Therefore, all pairs in $S'=\DFCL(i,S)$ start in $\{a_1,\ldots, a_{k+2}\}$, and there are fewer pairs in $S'$ of the
form $(a,b)$ where $b_{k+2}\prec b\prec b_1$.  As $C^*$ is contained
in $S'$, this contradicts the choice of $S$.  
\end{proof}

With Claim~\ref{clm:AB} in hand, we have shown that there is a set $S\in \NINR_{D3}$ with $\{a_1,a_2,\dots,a_{k+2}\}$ and $B(S)\subseteq\{b_1,b_2,\dots,b_{k+2}\}$ that contains the strict alternating cycle $C^*$. This set $S$ satisfies the hypotheses of Lemma~\ref{lem:size_contradiction}. Therefore, $|S|\leq (k+1)(k+2)/2+2-n$. However, throughout this subsection, we had assumed that each set in $\NINR$ had size greater than $(k+1)(k+2)/2+2-n$. The contradiction completes the proof of Theorem~\ref{thm:main-3}. Recall that this inequality is best possible as shown by the construction in Example~\ref{exa:k ge n-extreme}.

\section{Closing Remarks}\label{sec:close}

At the outset of the paper, we remarked that we were able to prove
Theorem~\ref{thm:chi=dim}
without settling the issue of whether $\alpha(G_n^k)=
(k+1)(k+2)/2$. To do so, we first obtained the full characterization of sets in $\MINR(n,k)$ when
$k<n\le 2k$.  This resulted in the proof that such sets have
size at most $2+(2k+1-n)(2k+2-n)/2$, a quantity which is less than $(k+1)(k+2)/2$.  

When $k\ge 2n-6$, it follows that $\dim(S_n^k)=3$. So to show that $\dim(S_n^k)=\chi(G_n^k)$
in this range, it is only necessary to prove that $G_n^k$ is not $2$-colorable.
However, this follows from a theorem of Cogis~\cite{bib:Cogi} which
asserts that if $P$ is a poset and $G$ is the associated graph of
critical pairs, then $\dim(P)\le 2$ if and only if $\chi(G)\le 2$.
This is a non-trivial theorem and the notation and terminology used
by Cogis are quite different from the style and contents of this
paper.  A more combinatorial proof of this result can be found in~\cite{bib:FelTro}.  Regardless, when $k\ge 2n-6$, it is an elementary exercise to show directly
that $\chi(G_n^k)\ge3$ by identifying an odd cycle contained in this
graph.

When $n\le k<2n-6$, we have $\dim(S_n^k)=4$.  We were able to show
that $\chi(G_n^k)=4$ in this range by showing that $G_n^k$ has so 
many triangles that no independent set could possibly intersect them all.  
This construction is a much more substantive exercise.

As for open problems, there are two obvious challenges stemming
from our work. We have not been able to completely determine
the sets which belong to $\MINR(n,k)$ when $n\le k$, and we are
not even certain that there is any reasonable way in which this
can be done. Also, when $k<n\le 2k$, it would be interesting to find 
a way to prove the inequality in Theorem~\ref{thm:main-2} without first
finding all sets in $\MINR(n,k)$.

As noted in several recent research papers, dimension can be
defined for a set of incomparable (or critical) pairs in a poset,
with the dimension of the poset being the special case
where we consider the entire set of incomparable pairs.  Accordingly,
it would be of interest to determine whether $\dim(S)=\chi(S)$
for subsets $S\subseteq\Inc(A,B)$.  We expect the answer to be
negative.

Finally, we should mention the classic problem of determining whether
there is some constant $c_0\ge3$ for which there is
an infinite sequence $\{P_n:n\ge3\}$ of posets such that for each
$n\ge3$, (1)~$\dim(P_n)\ge n$; and (2)~if $G_n$ is the
graph of critical pairs of $P_n$, then $\chi(G_n)\le c_0$.
If the answer is negative, then it
will become an interesting challenge to investigate classes of
posets for which the dimension can be bounded as a function
of the chromatic number of the associated graph of critical pairs.
 
\section{Acknowledgments}

The research reported here was initiated in two informal workshops
hosted by the United States Military Academy (West Point) in 2015 and
2016, and we thank them for providing a stimulating
and encouraging atmosphere.

Smith acknowledges support from NSF-DMS grant \#1344199. Harris acknowledges support from NSF-DMS grant \#1620202.

\end{document}